\documentclass{article}
\usepackage{hyperref}
\usepackage{amssymb}
\usepackage{amsmath,amsthm, latexsym, wasysym}
\usepackage{enumerate}
\usepackage{amsfonts}
\newtheorem{theorem}{Theorem}[section]
\newtheorem{lemma}[theorem]{Lemma}

\newtheorem{proposition}[theorem]{Proposition}
\newtheorem{remark}[theorem]{Remark}
\theoremstyle{definition}
\newtheorem{definition}[theorem]{Definition}

\numberwithin{equation}{section}
\usepackage[T1]{fontenc}
\usepackage{mathrsfs}
\usepackage{setspace}
\usepackage{graphicx,subfigure}
\usepackage{xcolor}
\usepackage{epstopdf}
\usepackage[utf8]{inputenc}
\usepackage{multicol}
\usepackage{amsmath}
\usepackage{amsthm}
\usepackage{upgreek}
\usepackage{amsfonts}
\usepackage{contour}
\usepackage{float}
\usepackage{breqn}
\usepackage{textcomp}
\usepackage{gensymb}
\usepackage[left=1.3in,right=1in,top=1in,bottom=1in]{geometry}
\usepackage{hyperref}
\usepackage{lipsum}
\usepackage{apptools}
\newcommand{\remove}[1]{}
\newcommand{\ra} {\rightarrow}
\newcommand{\RR} {\mathbb{R}}
\newcommand{\DD} {\displaystyle}
\newcommand{\la} {\lambda}
\theoremstyle{remark}

\begin{document}
\title {Nehari manifold  for
	fractional p(.)-Laplacian system involving  concave-convex nonlinearities}
	 \date{} 
 	\maketitle
 	\vspace{-1cm}
	\begin{center}
 {\large Reshmi Biswas and Sweta Tiwari\\
 Department of Mathematics, IIT Guwahati, Assam 781039, India}
\end{center}
\begin{abstract}
In this article using Nehari manifold method we study the multiplicity of solutions of the following nonlocal elliptic system involving variable exponents 
and  concave-convex nonlinearities:
\begin{equation*}
\;\;\;	\begin{array}{rl}
(-\Delta)_{p(\cdot)}^{s} u&=\lambda~ a(x)| u|^{q(x)-2}u+\frac{\alpha(x)}{\alpha(x)+\beta(x)}c(x)| u|^{\alpha(x)-2}u| v| ^{\beta(x)},\hspace{2mm}
x\in \Omega; \\
(-\Delta)_{p(\cdot)}^{s} v&=\mu~ b(x)| v|^{q(x)-2}v+\frac{\alpha(x)}{\alpha(x)+\beta(x)}c(x)| v|^{\alpha(x)-2}v| u| ^{\beta(x)},\hspace{2.5mm}
x\in \Omega; \\
u=v&=0 ,\hspace{1cm} x\in \Omega^c:=\mathbb R^N\setminus\Omega,
\end{array}
\end{equation*}
where $\Omega\subset\mathbb R^N,~N\geq2$ is a smooth bounded domain, $\lambda,\mu>0$ are the parameters,  $s\in(0,1),$ 
$p\in C(\RR^N\times \mathbb R^N,(1,\infty))$ and  $q,\alpha,\beta\in C(\overline{\Omega},(1,\infty))$ are the variable exponents and 
$a,b,c\in C(\overline{\Omega},[0,\infty))$ are the non-negative weight functions. 
We show that there exists $\Lambda>0$ such that for all $\lambda+\mu<\Lambda$, there exist two non-trivial and non-negative solutions of the 
above problem under some assumptions on $q,\alpha,\beta$.
\end{abstract}
\section{Introduction}
In this article we consider the following nonlocal elliptic system with variable exponents:
\begin{equation}\label{mainprob}
\;\;\;	\left.\begin{array}{rl}
(-\Delta)_{p(\cdot)}^{s} u&=\la~ a(x)| u|^{q(x)-2}u+\frac{\alpha(x)}{\alpha(x)+\beta(x)}c(x)| u|^{\alpha(x)-2}u| v| ^{\beta(x)},\hspace{2mm}
x\in \Omega, \\
(-\Delta)_{p(\cdot)}^{s} v&=\mu~ b(x)| v|^{q(x)-2}v+\frac{\alpha(x)}{\alpha(x)+\beta(x)}c(x)| v|^{\alpha(x)-2}v| u| ^{\beta(x)},\hspace{2.5mm}
x\in \Omega, \\
u=v&=0 ,\hspace{7.3cm} x\in \Omega^c:=\mathbb R^N\setminus\Omega,
\end{array}
\right\}
\end{equation}
where $\Omega\subset\RR^N,~N\geq2$ is a smooth bounded domain, $\la,\mu>0$ are the parameters, 
$s\in(0,1),$   $p\in C(\RR^N\times \RR^N,(1,\infty))$ with $sp^+<N$ and  $q,\alpha,\beta\in C(\overline{\Omega},(1,\infty))$ 
are the variable exponents and $a,b,c\in C(\overline{\Omega},[0,\infty))$ are the non-negative weight functions.
The nonlocal operator  $(-\Delta)_{p(\cdot)}^{s}$ is defined  as
\begin{equation}\label{operator}
(-\Delta)_{p(\cdot)}^{s} u(x):=  P.V.
\int_{\RR^N}\frac{\mid
	u(x)-u(y)\mid^{p(x,y)-2}(u(x)-u(y))}{\mid
	x-y\mid^{N+s(x,y)p(x,y)}}dy, ~~x \in \RR^N,
\end{equation}
\noindent where P.V. stands for Cauchy principle value.
Problems involving nonlocal operators have gained a lot of interests for research in recent years.
Mathematical modeling of problems in many areas like mechanics, population dynamics, thin obstacle problem,
optimization and finance involves fractional Laplacian $(-\Delta)^s$ or fractional $p$-Laplacian $(-\Delta)^s_p$.
We refer \cite{Bisci} and  \cite{HG} for the basic results on problems involving nonlocal operators.
Also one can refer \cite{parini,rad1, MoSq, SV2} and the references therein for the existence, multiplicity and regularity of the solutions
of these problems. \\
In present work our objective is to study the nonlocal elliptic problems with variable exponents. Operators involving variable growth are 
extensively studied due to the precision in modeling of various phenomenon where the property of the subject under consideration depends 
on the point of the observation,
for example in image restoration, electrorheological fluid and in non-Newtonian processes etc.
We refer \cite{ alves2,Die,Fan2,Fan1,KS} and references therein for the study of the problems involving $p(x)-$Laplace operator
defined as $\Delta_{p(x)}u:=\text{ div }(|\nabla u|^{p(x)-2}\nabla u)$.\\ Therefore there is a natural question  to see what results can be recovered
when the local $p(\cdot)$-Laplacian  is replaced by the fractional $p(\cdot)$-Laplacian.
The fractional Sobolev spaces with variable exponents and the corresponding fractional $p(\cdot)$-Laplace operator $(-\Delta)_{p(\cdot)}^{s}$   are introduced recently by
Kaufmann et al. in \cite{KRV}. Also in \cite{ ABS,AB,AR, ky ho} authors have established the basic properties of such spaces and studied the problems involving fractional $p(\cdot)$-Laplacian.\\
Using the Nehari manifold and the Fibering map, in case of local $p-$Laplacian, 
Brown and Wu (\cite{Wu}) have obtained multiple solutions of an 
elliptic system with sign changing weight functions and concave-convex nonlinearities. In nonlocal set-up, Sreenadh and Goyal (\cite{SG}) studied the same for the single fractional
$p-$ Laplacian equation .
Also we refer (\cite{chen-deng} ) where the authors studied the fractional p-Laplacian system involving concave-convex nonlinearities  via Nehari manifold and Fibering map.
In \cite{pucci} Pucci et al. modified the definition of Nehari manifold and Fibering map for the fractional $(p,q)-$Laplacian system and studied the corresponding Dirichlet problem.
Recently Alves et al. (\cite{alves})  used this Nehari manifold method to prove the multiplicity of solutions for $p(x)$-Laplacian 
problems in whole of $\mathbb R^N$.\\ 
Motivated by the above works, in this article we address the multiplicity of the solutions for the nonlocal elliptic system with variable exponents involving concave and 
convex nonlinearities using the analysis of the Fibering map and Nehari manifold. We note that the Nehari manifold approach through the Fibering map analysis for the functional
involving variable exponents
is interesting due to the non-homogeneity that arises from the variable exponents. But it is also worthy mentioning that due to the presence of the variable exponents most 
of the estimates do not hold immediately unlike in the constant exponent set-up. More precisely for  the non-homogeneity in the non-linear term and in the corresponding energy 
functional we loose some good properties which are valid in case of constant exponents. Hence in our present work we need to carry out some extra careful analysis to overcome this
issue.
According to our best of knowledge this the first work dealing with fractional $p(\cdot)-$Laplacian system involving concave and convex nonlinearities using Fibering-map approach.
\par  Next we set some notations as follows. Let  $\mathcal{D}$ be a domain.
For any function $\Phi:\mathcal{D}\rightarrow\mathbb R$, we set 
{\begin{align}\label{notation1}
	\Phi^{-}:=\inf_{\mathcal{D}} \Phi(x)\text{ ~~~and ~~} \Phi^{+}:=\sup_{ \mathcal{D}}\Phi(x).
	\end{align}}
\noindent We also define the function space
$$C_+(\mathcal{D}):=\{g\in C(\mathcal{D}, \RR):1 <g^{-}\leq g^{+}<\infty\}.$$
In order to state our result we assume that the variable exponents $p,q,\alpha~ and ~\beta$ and  the weight functions $a,b,c$ satisfy the
following hypotheses.  
\begin{itemize}
	\item[{\bf(P1).}] The variable exponent $p\in C_+(\RR^N\times\RR^N).$ 
	\item[{\bf(P2).}] The function $p$ is symmetric, i.e., $p(x,y)=p(y,x)$ for all $(x,y)\in \RR^N\times\RR^N$.	
	\item[\bf(A1).] The variable exponents $q,\alpha,\beta\in C_{+}(\overline{\Omega})$ and $p\in C_+(\RR^N\times\RR^N)$ satisfy the following:$$ 1<q^-\leq q^+<p^-\leq p^+<\alpha^-+\beta^-\leq\alpha^++\beta^+<p_s^{*-},$$ where $p^*_s(x)=\frac{Np(x,x)}{N-sp(x,x)},~ x\in \overline{\Omega}$ is the critical exponent.
	\item[\bf{(A2).}] We also assume $\DD\frac{p^-}{\alpha^++\beta^+}<\frac{p^--q^+}{\alpha^++\beta^+-q^+}. \frac{\alpha^-+\beta^--q^-}{p^+-q^-}.$ 
	\item[\bf{(A3).}] The non-negative weight functions $a,b\in L^{q_*(x)}(\Omega),$ where $$q_*(x)=\DD\frac{\alpha(x)+\beta(x)}{\alpha(x)+\beta(x)-q(x)}.$$ 
	\item[\bf{(A4).}] The non-negative weight function $c\in L^\infty(\Omega).$
\end{itemize}
\begin{remark}
	$(A2)$  is equivalent to the condition $0<p<\alpha+\beta$ when all the exponents are constants.
\end{remark}
Now we define the weak solution of problem \eqref{mainprob} in the functional space $E$, defined in  Section 2, as follows:
\begin{definition}\label{defi}
	We say that $(u,v)\in E$ is a weak solution of the problem \eqref{mainprob}, if  we have 
	{	\begin{align}\label{weakform}
		&\int_{\RR^N\times\RR^N}\frac{|
			u(x)-u(y)|^{p(x,y)-2}(u(x)-u(y))(\phi(x)-\phi(y))}{|
			x-y|^{N+sp(x,y)}}dxdy\nonumber\\& + \int_{\RR^N\times\RR^N}\frac{|
			v(x)-v(y)|^{p(x,y)-2}(v(x)-v(y))(\psi(x)-\psi(y))}{|
			x-y|^{N+sp(x,y)}}dxdy\nonumber\\&=\int_{\Omega}\Big(\la a(x)| u|^{q(x)-2}u\phi+
		\mu b(x)| v|^{q(x)-2}v\psi\Big) dx\nonumber\\ &+\int_{\Omega}\frac{\alpha(x)}{\alpha(x)+\beta(x)}c(x)| u|^{\alpha(x)-2}u| v| ^{\beta(x)}\phi dx\nonumber\\&+\int_{\Omega}\frac{\beta(x)}{\alpha(x)+\beta(x)}c(x)| v|^{\alpha(x)-2}v| u| ^{\beta(x)}\psi dx~~~~\text {for all } (\phi,\psi)\in E.
		\end{align} } 
\end{definition}
The main result in this article is stated as follows:
\begin{theorem}\label{mainthm}
	Let $\Omega\subset\RR^N, N\geq2$ be a smooth bounded domain, $s\in(0,1)$ and $p(\cdot,\cdot)$ satisfy $(P1)-(P2)$ with $sp^+<N.$
	Assume the hypotheses $(A1)-(A4)$ hold true. Then there exists a positive constant
	$\Lambda=\Lambda(N,s,p,q,\alpha,\beta,a,b,c,\Omega)$ such that for any pair of positive parameters $(\la,\mu)$ with $\la+\mu<\Lambda,$
	the problem \ref{mainprob} has at least two non-trivial, non-negative weak solutions.
\end{theorem}
\section{ Preliminary results}
Here we recall the variable exponent Lebesgue spaces. For more details regarding these space one can refer {\cite{Die, Fan1} and references therein.}\\
For $\gamma \in C_{+}(\overline{\Omega}), $
we define the following variable exponent Lebesgue space:
$$L^{\gamma(x)}(\Omega)=\Big\{u:\Omega\ra\RR~ \text {is~measurable}:\int_{\Omega}|u|^{\gamma(x)}<+\infty\Big\},$$
This space  is a separable, reflexive Banach space equipped with the following Luxemburg norm:\\
$$\parallel u \parallel_{L^{\gamma(x)}(\Omega)} = \inf\Big\{\eta>0:\int_{\Omega}\Big|\frac{ u}{\eta}\Big|^{\gamma(x)}\leq 1\Big\}.$$
The space $(L^{\gamma(x)}(\Omega),\parallel \cdot \parallel_{L^{\gamma(x)}(\Omega)})$ is a separable, reflexive Banach space.\\
We also recall the following H\"older-type inequality.
\begin{lemma}\label{holder}
	Let $\gamma'\in C_+(\overline{\Omega})$ such that  $ \DD\frac{1}{\gamma(x)} + \frac {1}{\gamma'(x)}=1$.
	Then for any  $u \in L^{\gamma(x)}(\Omega)$ and $ v\in L^{\gamma'(x)}(\Omega)$ we have
	$$\left| \int_{\Omega} uv dx \right|\leq \Big(\frac{1}{\gamma^{-}} + \frac{1}{\gamma^{'-}} \Big) \parallel u \parallel_{L^{\gamma(x)}(\Omega)}
	\parallel v \parallel_{L^{\gamma'(x)}(\Omega)}. $$
\end{lemma}
\noindent Next we recall { Lemma A.1 in \cite{JSG}} for variable exponent Lebesgue spaces.
\begin{lemma}\label{lemA1}
	Let $\nu_1(x)\in L^\infty(\Omega)$ such that $\nu_1\geq0,\; \nu_1\not\equiv 0.$ Let $\nu_2:\Omega\ra\RR$ be a measurable function 
	such that $\nu_1(x)\nu_2(x)\geq 1$ a.e. in $\Omega.$ Then for every $u\in L^{\nu_1(x)\nu_2(x)}(\Omega),$ 
	$$\parallel |u|^{\nu_1(\cdot)}\parallel_{L^{\nu_2(x)}(\Omega)}\leq 
	\parallel u\parallel_{L^{\nu_1(x)\nu_2(x)}(\Omega)}^{\nu_1^-}+\parallel u\parallel_{L^{\nu_1(x)\nu_2(x)}(\Omega)}^{\nu_1^+}.$$
\end{lemma}	
\noindent  The modular  $\rho_{\gamma}:L^{\gamma(x)}(\Omega)\ra \RR$ is defined as
$$\rho_{\gamma}(u) = \int_\Omega |u|^{\gamma(x)}dx.$$
We  also state the following result from \cite{Fan1} where the authors established the relations between norm
$\parallel \cdot \parallel_{L^{\gamma(x)}(\Omega)}$ and 
the corresponding modular function $\rho_{\gamma}(\cdot)$ as follows:
\begin{lemma}\label{prp1}
	Let $u \in L^{\gamma(x)}(\Omega)$, then 
	\begin{enumerate}
		\item[(i)]$ \parallel u \parallel_{L^{\gamma(x)}(\Omega)}<1(=1;>1) \text{~ if and only if ~} \rho_{\gamma}(u)<1(=1;>1);$
		\item[(ii)] If $\parallel u \parallel_{L^{\gamma(x)}(\Omega)}>1$, then $\parallel u \parallel_{L^{\gamma(x)}(\Omega)} ^{\gamma^{-}}\leq\rho_{\gamma}(u)\leq\parallel u \parallel_{L^{\gamma(x)}(\Omega)}^{\gamma^{+}}$ ;
		\item[(iii)] If $\parallel u \parallel_{L^{\gamma(x)}(\Omega)}<1$, then $\parallel u \parallel_{L^{\gamma(x)}(\Omega)}^{\gamma^{+}}\leq\rho_{\gamma}(u)\leq\parallel u \parallel_{L^{\gamma(x)}(\Omega)}^{\gamma^{-}}.$
	\end{enumerate}
\end{lemma}
\begin{lemma}\label{prp2}
	Let $u,u_{m}  \in L^{{\gamma(x)}}(\Omega),~m=1,2,3,\cdots $. Then the following statements are equivalent:
	\begin{enumerate}
		\item[(i)] $\DD{\lim_{m\ra \infty} }\parallel u_{m} - u \parallel_{L^{\gamma(x)}} =0;$
		\item[(ii)] $\DD{\lim_{m\ra \infty}} \rho_{\gamma}(u_{m} -u)=0;$
		\item[(iii)] $ u_{m} \text{ converges to $u$ in~ }  \Omega  \text { ~in measure and ~} \DD{\lim_{m\ra \infty}}  \rho_{\gamma}(u_{m})= \rho_{\gamma}(u).$
	\end{enumerate}
\end{lemma}
\subsection{Fractional Sobolev spaces with variable exponents}
In this section, we discuss the properties of fractional Sobolev spaces with  variable exponents. These spaces have been 
introduced in \cite {KRV} for the first time. Also in \cite{AB,AR,ky ho} the authors established some important  properties of these spaces. 
\par Let $\Omega\subset\RR^N$ be a smooth bounded domain
and $p(\cdot, \cdot)$ satisfy  $(P1)-(P2)$. We denote$$ \overline{p}(x)=p(x,x)~ \text{for any~} x\in\RR^N.$$
Thus $\overline{p}\in C_{+}(\overline{\Omega}).$ Now we define
the  fractional Sobolev 
space with variable exponents as follows:
{\begin{align*}
	&{ W}=W^{s,\overline{p}(x),p(x,y)}(\Omega)\\
	&:=\bigg \{ u\in L^{ \overline{p}(x)}(\Omega):
	\int_{\Omega\times\Omega}\frac{\mid u(x)-u(y)\mid^{p(x,y)}}{\eta^{p(x,y)}\mid x-y \mid^{N+s(x,y)p(x,y)}}dxdy<\infty,
	\text{ for some }\eta>0\bigg\}.
	\end{align*}}
We set the  seminorm as:
$$[u]_{\Omega}^{s, p(x,y)}:=\DD\inf \left\{\eta >0 : \int_{\Omega\times\Omega}\frac{\mid
	u(x)-u(y)\mid^{p(x,y)}}{\eta^{p(x,y)}\mid x-y \mid^{N+s(x,y)p(x,y)}}dxdy<1\right\}.$$
\noindent Then $({W}, \|\cdot\|_{{W}})$ is a separable reflexive Banach space {(see \cite{AR})} equipped with the norm 
\begin{align*}
\| u \|_{\overline{W}}:= \| u \|_{L^{\overline{p}(x)}(\Omega)}+
[u]_{\Omega}^{s, p(x,y)}.
\end{align*}
\noindent We state the following continuous and compact embedding theorem for ${ W}$ as studied in \cite{ky ho}.
\begin{theorem}\label{embd}
	Let $\Omega$ be a smooth bounded domain in $\mathbb{R}^N$, $s\in(0,1)$ and $p(\cdot, \cdot)$ satisfied $(P1)-(P2)$ with $sp^+<N.$  
	Let $r\in C_{+}(\overline{\Omega})$ 
	such that $ 1<r^{-}\leq r(x)<p_s^{*}(x)=\frac{N\tilde{p}(x)}{N-s\tilde{p}}$
	for $x\in\overline \Omega$. Then,
	there exits a constant $C=C(N,s,p,r,\Omega)>0$
	such that,
	for any $u\in {W}$,
	$$
	\| u \|_{L^{r(x)}(\Omega)}\leq K \| u \|_{{ W}}.$$			 			
	Moreover, this embedding is compact. 	 
\end{theorem}
\noindent For studying nonlocal problems 
involving the operator $(-\Delta)_{p(\cdot)}^{s}$ with Dirichlet boundary datum via variational methods, we define another new fractional type 
Sobolev spaces with variable exponents. 
One can refer {\cite{Bisci} }and references therein for this type of spaces in fractional $p$-Laplacian framework.
We set $\mathcal{Q}:=\mathbb R^{2N}\setminus(\Omega^c\times\Omega^c)$ and define the new fractional Sobolev space with variable exponent as:
{\begin{align*}
	{X}&=X^{s,\overline{p}(x),p(x,y)}(\Omega)\\& :=\bigg\{u:\RR^N\rightarrow\mathbb R:u_{|_\Omega}\in L^{\overline{p}(x)}(\Omega),\\
	&\;\;\;\;\;\;\;\;\int_{\mathcal{Q}}\frac{|u(x)-u(y)|^{p(x,y)}}{\eta^{p(x,y)}| x-y|^{N+sp(x,y)}}dxdy<\infty,\text{ for some } \eta>0\bigg\}.
	\end{align*}}
The space ${X}$ is equipped with the following norm:
{$$\| u\|_{ {X}}:=\| u\|_{L^{\overline{p}(x)}(\Omega)}+\inf\Big\{\eta>0:\int_Q\frac{\mid
		u(x)-u(y)\mid^{p(x,y)}}{\eta^{p(x,y)}\mid x-y \mid^{N+sp(x,y)}}dxdy<1\Big\},$$} where $[u]_{X}$ is the seminorm, defined as
$$[u]_{X}=\inf\Big\{\eta>0:\int_{\mathcal{Q}}\frac{\mid
	u(x)-u(y)\mid^{p(x,y)}}{\eta^{p(x,y)}\mid x-y \mid^{N+sp(x,y)}}dxdy<1\Big\}.$$
Then $(X,\|\cdot\|_{X})$ is a separable reflexive Banach space.
Next we define the subspace ${X}_0$ of ${ X}$ as 
$${ X_0}={ X}_0^{s,\overline{p}(x),p(x,y)}(\Omega):=\{u\in { X}\;:\; u=0\;a.e.\; in\;\Omega^c\}.$$
We define the norm on ${ X}_0$ as follows:
\begin{align*}
\| u\|_{{ X}_0}:=\inf\Big\{\eta>0:\int_{\mathcal{Q}}\frac{\mid
	u(x)-u(y)\mid^{p(x,y)}}{\eta^{p(x,y)}\mid x-y \mid^{N+sp(x,y)}}dxdy<1\Big\}.
\end{align*}
\begin{remark}\label{rem1}
	For $u\in X_0,$ we get
	$$\int_{\mathcal{Q}}\frac{\mid
		u(x)-u(y)\mid^{p(x,y)}}{\eta^{p(x,y)}\mid x-y \mid^{N+sp(x,y)}}dxdy= \int_{\RR^N\times\RR^N}\frac{\mid
		u(x)-u(y)\mid^{p(x,y)}}{\eta^{p(x,y)}\mid x-y \mid^{N+sp(x,y)}}dxdy.$$ Thus we have
	\begin{align*}
	\| u\|_{{ X}_0}:=\inf\Big\{\eta>0:\int_{\RR^N\times\RR^N}\frac{\mid
		u(x)-u(y)\mid^{p(x,y)}}{\eta^{p(x,y)}\mid x-y \mid^{N+s(x,y)p(x,y)}}dxdy<1\Big\}.
	\end{align*}  
\end{remark}
Now we have the following 
continuous and compact embedding result for the space ${ X}_0$. 
The proof follows from the Theorem 2.2 and Remark 2.2 in \cite{ABS}.
\begin{theorem}\label{prp 3.3}
	Let $\Omega $ be a smooth bounded domain in $\mathbb{R}^N $ and let $s\in(0,1).$ Let $p(\cdot, \cdot)$  satisfy $(P1)-(P2)$  with $sp^+<N$.
	Then for any  $r\in C_+(\overline{\Omega})$ such that $1<r(x)< p_s^*(x)$ for all $x\in \overline{\Omega}$,
	there exits a constant $C=C(N,s,p,r,\Omega)>0$
	such that for every $u\in{ X}_0$, 
	\begin{align*}
	\| u \|_{L^{r(x)}(\Omega)}\leq C \| u \|_{{ X}_0}.
	\end{align*}
	Moreover, this embedding is
	compact.
\end{theorem}
\begin{definition}
	For $u\in { X}_0$, we define the  modular  $\rho_{{ X}_0}:{ X}_0\ra \RR$ as follows:
	\begin{align}\label{modular}
	\rho_{{ X}_0}(u):=\int_{\RR^n \times \RR^n}\frac{\mid
		u(x)-u(y)\mid^{p(x,y)}}{\mid x-y \mid^{N+s(x,y)p(x,y)}}dxdy.\end{align}
\end{definition}
The interplay between the norm in ${ X}_0$ and the modular function $\rho_{{ X}_0}$ can be studied in the following lemma:		 	                                                       
\begin{lemma}\label{lem 3.1}
	Let $u \in { X}_0$ and $\rho_{{ X}_0}$ be  defined as in \eqref{modular}. Then we have the following results:
	\begin{enumerate}
		\item[$(i)$]$ \| u \|_{{ X}_0}<1(=1;>1)\text {~if and only if~} \rho_{{ X}_0}(u)<1(=1;>1).$
		\item[$(ii)$] If $\| u \|_{{ X}_0}>1$, then
		$
		\| u \|_{{ X}_0} ^{p^{-}}\leq\rho_{{ X}_0}(u)\leq\| u \|_{{ X}_0}^{p{+}}.
		$
		\item[$(iii)$] If $\| u \|_{{ X}_0}<1$, then 
		$\| u \|_{{ X}_0} ^{p^{+}}\leq\rho_{{ X}_0}(u)\leq\| u \|_{{ X}_0}^{p{-}}.
		$
	\end{enumerate}
\end{lemma}	The next lemma can easily be obtained using the properties of the modular function $\rho_{X_0}$ in Lemma \ref {lem 3.1}.	 	
\begin{lemma}\label{lem 3.2}
	Let $u,u_{m}  \in { X}_0$, $m\in\mathbb N$. Then the following statements are equivalent:
	\begin{enumerate}
		\item[$(i)$] $\DD{\lim_{m\ra \infty} }\| u_{m} - u \|_{{ X}_0} =0,$
		\item[$(ii)$] $\DD{\lim_{m\ra \infty}} \rho_{{ X}_0}(u_{m} -u)=0.$
	\end{enumerate}
\end{lemma}
\begin{lemma}{$($\cite{ABS}$)$}\label{lem 3.3}
	$(X_0,\|\cdot\|_{X_0})$ is a separable, reflexive and uniformly convex Banach space.
\end{lemma}
\begin{remark}
	We define $E:=X_0\times X_0$ as the solution space corresponding to our problem \eqref{mainprob}, equipped with the norm $\|(u,v)\|=\max\{\|u\|_{X_0}, \|v\|_{X_0}\}.$ Clearly $(E, \|(\cdot, \cdot)\|)$ is a reflexive, separable Banach space. 
\end{remark}
\section{ Nehari manifold and Fibering map analysis} 
Here first we discuss certain technical results regarding the Nehari manifold and the 
Fibering map and the behavior of the energy functional corresponding to the problem \eqref{mainprob}.
The Euler functional $J_{\la,\mu}:E\ra\RR$ associated to the problem \eqref{mainprob} is defined as  
{\small	\begin{align}\label{energy}
	J_{\la,\mu}(u,v)	&=\int_{\RR^N\times\RR^N}\frac{1}{p(x,y)}\frac{|
		u(x)-u(y)|^{p(x,y)}}{|
		x-y|^{N+sp(x,y)}}dxdy + \int_{\RR^N\times\RR^N}\frac{1}{p(x,y)}\frac{|
		v(x)-v(y)|^{p(x,y)}}{|
		x-y|^{N+sp(x,y)}}dxdy\nonumber\\ &-\int_{\Omega} \frac{1}{q(x)} \Big(\la a(x)| u|^{q(x)} +
	\mu b(x)| v|^{q(x)} \Big) dx \nonumber\\&-\int_{\Omega}\frac{1}{\alpha(x)+\beta(x)}c(x)| u|^{\alpha(x)}| v| ^{\beta(x)} dx.
	\end{align} }
By a direct computation it easily follows that  $J_{\la,\mu}\in C^{1}(E,\RR) $  and 
{\small	\begin{align*}
	\langle J'_{\la,\mu}(u,v), (\phi,\psi)\rangle	&=\int_{\RR^N\times\RR^N}\frac{|
		u(x)-u(y)|^{p(x,y)-2}(u(x)-u(y))(\phi(x)-\phi(y))}{|
		x-y|^{N+sp(x,y)}}dxdy \nonumber\\&+ \int_{\RR^N\times\RR^N}\frac{|
		v(x)-v(y)|^{p(x,y)-2}(v(x)-v(y))(\psi(x)-\psi(y))}{|
		x-y|^{N+sp(x,y)}}dxdy\nonumber\\&-\int_{\Omega}\Big(\la a(x)| u|^{q(x)-2}u\phi+
	\mu b(x)| v|^{q(x)-2}v\psi\Big) dx\nonumber\\ &-\int_{\Omega}\frac{\alpha(x)}{\alpha(x)+\beta(x)}c(x)| u|^{\alpha(x)-2}u| v| ^{\beta(x)}\phi dx\nonumber\\&-\int_{\Omega}\frac{\beta(x)}{\alpha(x)+\beta(x)}c(x)| v|^{\alpha(x)-2}v| u| ^{\beta(x)}\psi dx~~~~\text {for any } (\phi,\psi)\in E.
	\end{align*} } 
Therefore, the weak solutions of \ref{mainprob} are critical points of the functional $J_{\la,\mu}.$
One can note that $J_{\la,\mu}$ is not bounded below on $E,$ but it is bounded below on the following subset of $E$. 
We define the Nehari manifold as
$$\mathscr {N}_{\la,\mu}:=\{(u,v)\in E\setminus\{(0,0)\}~:~ \langle J'_{\la,\mu}(u,v), (u,v)\rangle=0 \}.$$
Therefore, $(u,v)\in\mathscr {N}_{\la,\mu}$ if and only if
\begin{align}\label{1.0}
&\int_{\RR^N\times\RR^N}\frac{|
	u(x)-u(y)|^{p(x,y)}}{|
	x-y|^{N+sp(x,y)}}dxdy + \int_{\RR^N\times\RR^N}\frac{|
	v(x)-v(y)|^{p(x,y)}}{|
	x-y|^{N+sp(x,y)}}dxdy\nonumber\\&-\int_{\Omega}\Big(\la a(x)| u|^{q(x)}+
\mu b(x)| v|^{q(x)}\Big) dx -\int_{\Omega}c(x)| u|^{\alpha(x)}| v| ^{\beta(x)} dx=0.
\end{align}
The Nehari manifold closely associated to the behavior of the function $\varphi_{u,v}:\RR^+\ra\RR$ for a given $(u,v)\in E$, defined 
as $\varphi_{u,v}(t)=J_{\la,\mu}(tu,tv).$
This map is called Fibering maps and was introduced by Drabek and Pohozaev in \cite{DP} and are also discussed in \cite{BZ} and \cite{SG}.
For $(u,v)\in E$, we have
{\small\begin{align}\label{1.1}
\varphi_{u,v}(t)=J_{\la,\mu}(tu,tv)&=\int_{\RR^N\times\RR^N}\frac{t^{p(x,y)}}{p(x,y)}\bigg\{\frac{|
	u(x)-u(y)|^{p(x,y)}}{|
	x-y|^{N+sp(x,y)}}+\frac{|
	v(x)-v(y)|^{p(x,y)}}{|
	x-y|^{N+sp(x,y)}}\bigg\}dxdy\nonumber\\ &-\int_{\Omega}\frac{t^{q(x)}}{q(x)}\Big(\la a(x)| u|^{q(x)}+
\mu b(x)| v|^{q(x)}\Big) dx \nonumber\\ &-\int_{\Omega}\frac{t^{\alpha(x)+\beta(x)}}{\alpha(x)+\beta(x)}c(x)| u|^{\alpha(x)}| v| ^{\beta(x)} dx.
\end{align} }
{\small\begin{align}\label{1.2}
\varphi'_{u,v}(t)&=\langle J'_{\la,\mu}(tu,tv), (u,v)\rangle\nonumber\\ &=\int_{\RR^N\times\RR^N}{t^{p(x,y)-1}}\bigg\{\frac{|
	u(x)-u(y)|^{p(x,y)}}{|
	x-y|^{N+sp(x,y)}}+\frac{|
	v(x)-v(y)|^{p(x,y)}}{|
	x-y|^{N+sp(x,y)}}\bigg\}dxdy\nonumber\\ &-\int_{\Omega}{t^{q(x)-1}}\Big(\la a(x)| u|^{q(x)}+
\mu b(x)| v|^{q(x)}\Big) dx -\int_{\Omega}{t^{\alpha(x)+\beta(x)-1}}c(x)| u|^{\alpha(x)}| v| ^{\beta(x)} dx.
\end{align}}
{\small\begin{align}\label{1.3}
\varphi''_{u,v}(t) &=\int_{\RR^N\times\RR^N}(p(x,y)-1){t^{p(x,y)-2}}\bigg\{\frac{|
	u(x)-u(y)|^{p(x,y)}}{|
	x-y|^{N+sp(x,y)}}+\frac{|
	v(x)-v(y)|^{p(x,y)}}{|
	x-y|^{N+sp(x,y)}}\bigg\}dxdy\nonumber\\ &-\int_{\Omega}(q(x)-1){t^{q(x)-2}}\Big(\la a(x)| u|^{q(x)}+
\mu b(x)| v|^{q(x)}\Big) dx \nonumber\\ &-\int_{\Omega}(\alpha(x)+\beta(x)-1){t^{\alpha(x)+\beta(x)-2}}c(x)| u|^{\alpha(x)}| v| ^{\beta(x)} dx.
\end{align} } 
Then using the fact that $\varphi'_{u,v}(t)=\langle J'_{\la,\mu}(tu,tv), (u,v)\rangle,$ we can see that $(tu,tv)\in \mathscr N_{\la,\mu}$ 
if and only if  $\varphi'_{u,v}(t)=0,$ that is, in particular $(u,v)\in \mathscr N_{\la,\mu}$ if and only if $\varphi'_{u,v}(1)=0.$ 
Thus it is natural to split $\mathscr N_{\la,\mu}$ into three parts corresponding to local maxima, local minima and points of 
inflection of the  function $\varphi_{u,v}$ as followings:
{\small$$\mathscr N^+_{\la,\mu}:=\{(u,v)\in \mathscr N_{\la,\mu} ~:~\varphi''_{u,v}(1)>0\}=\{(tu,tv)\in E\setminus\{(0,0)\} ~:~ \varphi'_{u,v}(t)=0,~\varphi''_{u,v}(1)>0\};$$
	$$\mathscr N^-_{\la,\mu}:=\{(u,v)\in \mathscr N_{\la,\mu} ~:~\varphi''_{u,v}(1)<0\}=\{(tu,tv)\in E\setminus\{(0,0)\} ~:~ \varphi'_{u,v}(t)=0,~\varphi''_{u,v}(1)<0\};$$
	$$\mathscr N^0_{\la,\mu}:=\{(u,v)\in \mathscr N_{\la,\mu} ~:~\varphi''_{u,v}(1)=0\}=\{(tu,tv)\in E\setminus\{(0,0)\} ~:~ \varphi'_{u,v}(t)=0,~\varphi''_{u,v}(1)=0\}.$$ }
Hence for any $(u,v)\in \mathscr N_{\la,\mu},$ from \eqref{1.0}, \eqref{1.2} and \eqref{1.3}, we deduce
{\small\begin{align}\label{00}
\varphi''_{u,v}(1) &=\int_{\RR^N\times\RR^N}p(x,y)\bigg\{\frac{|
	u(x)-u(y)|^{p(x,y)}}{|
	x-y|^{N+sp(x,y)}}+\frac{|
	v(x)-v(y)|^{p(x,y)}}{|
	x-y|^{N+sp(x,y)}}\bigg\}dxdy\nonumber\\ &-\int_{\Omega}q(x)\Big(\la a(x)| u|^{q(x)}+
\mu b(x)| v|^{q(x)}\Big) dx \nonumber\\ &-\int_{\Omega}(\alpha(x)+\beta(x))c(x)| u|^{\alpha(x)}| v| ^{\beta(x)} dx.
\end{align}}
\noindent For a given pair of functions $(u,v)\in E,$ we set {\small$${P}(u,v):=\DD\int_{\RR^N\times\RR^N}\bigg\{\frac{|
	u(x)-u(y)|^{p(x,y)}}{|
	x-y|^{N+sp(x,y)}}+\frac{|
	v(x)-v(y)|^{p(x,y)}}{|
	x-y|^{N+sp(x,y)}}\bigg\}dxdy,$$ $${Q}(u,v):=\DD\int_{\Omega}\Big(\la a(x)| u|^{q(x)}+
\mu b(x)| v|^{q(x)}\Big) dx$$ and $${R}(u,v):=\DD\int_{\Omega}c(x)| u|^{\alpha(x)}| v| ^{\beta(x)} dx.$$ }
In the next lemma we obtain some estimations on $P, Q$ and $R.$
\begin{lemma}\label{lem4} Let $(u,v)\in E.$ Then we have the followings:
	\begin{itemize}
		\item [$(i).$]  $ \begin{cases}
		\|(u,v)\|^{p^+},& \|(u,v)\|<1 \\ 
		\|(u,v)\|^{p^-},& \|(u,v)\|>1 
		\end{cases}\leq {P}(u,v) \leq\begin{cases}
		2\|(u,v)\|^{p^-},& \|(u,v)\|<1 \\ 
		2\|(u,v)\|^{p^+},& \|(u,v)\|>1 .
		\end{cases}$
		\item [$(ii).$] There exists a constant $C_1=C_1(N,s,p,q,\alpha,\beta,a,b,\Omega)>0$ such that{\small $${Q}(u,v)\leq C_1(\la+\mu) \max\{\|(u,v)\|^{q^-},~\|(u,v)\|^{q^+}\}.$$}
		\item [$(iii).$] There exists a constant $C_2=C_2(N,s,p,\alpha,\beta,c,\Omega)>1$ such that {\small$$ {R}(u,v)\leq C_2 \max\{\|(u,v)\|^{r^-},~\|(u,v)\|^{r^+}\}.$$}		
	\end{itemize}	
\end{lemma}
\begin{proof}
	$(i.)$ Clearly ${P}(u,v)=\rho_{X_0}(u)+\rho_{X_0}(v).$ Hence, we have 
	\begin{align}\label{1.4}
	\max\{\rho_{X_0}(u),\rho_{X_0}(v)\}\leq P(u,v)\leq 2\max\{\rho_{X_0}(u),\rho_{X_0}(v)\}
	\end{align}
	For $\|(u,v)\|>1,$ there are two cases.\\
	\underline{Case $I$}. $\|u\|_{X_0}>1$  and $\|v\|_{X_0}>1 :$ Then from {Lemma \ref{lem 3.1} }, we have
	\begin{align}\label{1.5}
	\|u\|_{X_0}^{p^-}<\rho_{X_0}(u)<\|u\|_{X_0}^{p^+}~~\text{ and }~~ \|v\|_{X_0}^{p^-}<\rho_{X_0}(v)<\|v\|_{X_0}^{p^+}.
	\end{align}
	Thus from \eqref{1.4} and \eqref{1.5}, we get 
	{\small$${P}(u,v)\leq 2\DD\max \{\|u\|_{X_0}^{p^+},~\|v\|_{X_0}^{p^+}\}=2\|(u,v)\|^{p^+}; 
		{P}(u,v)\geq \DD\max \{\|u\|_{X_0}^{p^-},~\|v\|_{X_0}^{p^-}\}=\|(u,v)\|^{p^-}.$$}
	\underline{Case $II$}. Without loss of generality, let $\|v\|_{X_0}<1<\|u\|_{X_0}$: Then we have $\|(u,v)\|=\|u\|_{X_0}.$ 
	Now from { Lemma \ref{lem 3.1}}, we get
	\begin{align}\label{1.6}
	\|u\|_{X_0}^{p^-}<\rho_{X_0}(u)<\|u\|_{X_0}^{p^+}~~ and ~~ \|v\|_{X_0}^{p^+}<\rho_{X_0}(v)<\|v\|_{X_0}^{p^-}.
	\end{align}
	From \eqref{1.4} and \eqref{1.6}, we have 
	{\small$${P}(u,v)\leq 2\max \{\|u\|_{X_0}^{p^+}, \|v\|_{X_0}^{p^+}\}=2\|(u,v)\|^{p^+}\text{~and~} {P}(u,v)\geq \max \{\|u\|_{X_0}^{p^-}, \|v\|_{X_0}^{p^-}\}=\|(u,v)\|^{p^-}.$$}
	Again for $\|(u,v)\|<1$, we have $\|u\|_{X_0}<1~ and~ \|v\|_{X_0}<1.$ Applying { Lemma \ref{lem 3.1} },
	we obtain 
	\begin{align}\label{1.7}
	\|u\|_{X_0}^{p^+}<\rho_{X_0}(u)<\|u\|_{X_0}^{p^-}~~ and ~~ \|v\|_{X_0}^{p^+}<\rho_{X_0}(v)<\|v\|_{X_0}^{p^-}.
	\end{align}
	Hence using \eqref{1.4} and \eqref{1.7}, we deduce
	{\small$${P}(u,v)\leq 2\max \{\|u\|_{X_0}^{p^-}, \|v\|_{X_0}^{p^-}\}=2\|(u,v)\|^{p^-}; {P}(u,v)\geq \max \{\|u\|_{X_0}^{p^+}, \|v\|_{X_0}^{p^+}\}=\|(u,v)\|^{p^+}. $$}
	Thus we get $(i).$\\
	$(ii).$ Using H\"{o}lder's inequality {(Lemma \ref{holder}), Sobolev-type embedding (Lemma \ref{prp 3.3})} and Lemma \ref{lemA1}, we have
	{\small\begin{align*}
		{Q}(u,v)&=\DD\int_{\Omega}\Big(\la a(x)| u|^{q(x)}+
		\mu b(x)| v|^{q(x)}\Big) dx\nonumber\\
		&\leq2\la \|a\|_{L^{q_*(x)}(\Omega)}\||u|^{q(\cdot)}\|_{L^{\frac{\alpha(x)+\beta(x)}{q(x)}}(\Omega)}+2\mu\|b\|_{L^{q_*(x)}(\Omega)}\||v|^{q(\cdot)}\|_{L^{\frac{\alpha(x)+\beta(x)}{q(x)}}(\Omega)}\nonumber\\
		&\leq 2\la \|a\|_{L^{q_*(x)}(\Omega)} \Big\{\|u\|_{L^{\alpha(x)+\beta(x)}(\Omega)}^{q^-}+\|u\|_{L^{\alpha(x)+\beta(x)}(\Omega)}^{q^+}\Big\}\nonumber\\
		&+2\mu \|b\|_{L^{q_*(x)}(\Omega)} \Big\{\|v\|_{L^{\alpha(x)+\beta(x)}(\Omega)}^{q^-}+\|v\|_{L^{\alpha(x)+\beta(x)}(\Omega)}^{q^+}\Big\}\nonumber\\
		&\leq K_1\Big[ \la \Big\{\|u\|_{X_0}^{q^-}+\|u\|_{X_0}^{q^+}\Big\}
		+\mu  \Big\{\|v\|_{X_0}^{q^-}+\|v\|_{X_0}^{q^+}\Big\}\Big]\nonumber\\
		&\leq C_1(\la+\mu) \max \Big\{ \|u\|_{X_0}^{q^-}, \|u\|_{X_0}^{q^+},~ \|v\|_{X_0}^{q^-},~ \|v\|_{X_0}^{q^+}\Big\}\nonumber\\
		&= C_1(\la+\mu) \max \Big\{\max \Big\{\|u\|_{X_0}^{q^-},~ \|v\|_{X_0}^{q^-}\Big\},~ \max \Big\{\|u\|_{X_0}^{q^+},~ \|v\|_{X_0}^{q^+}\Big\}\Big\}\nonumber\\
		&= C_1(\la+\mu) \max \Big\{\|(u,v)\|^{q^-},~ \|(u,v)\|^{q^+}\Big\}, \text{~where}
		\end{align*}}
	{\small  $K_1=2\Big\{\|a\|_{L^{q_*(x)}(\Omega)}+ \|b\|_{L^{q_*(x)}(\Omega)}\Big\}.$
		$\max\Big\{\Big(C( N,s,p,\alpha,\beta,\Omega)\Big)^{q^-}, \Big(C( N,s,p,\alpha,\beta,\Omega)\Big)^{q^+}\Big\}$} \\and $C_1=4K_1.$\\
	$(iii).$ Using Young's inequality, Lemma \ref{lemA1} and Lemma \ref{prp 3.3}, we have
		{\small\begin{align*}
 R(u,v)&=\DD\int_{\Omega}c(x)| u|^{\alpha(x)}| v| ^{\beta(x)} dx\nonumber\\
	&\leq \|c\|_{L^{\infty}(\Omega)}\DD\int_{\Omega}| u|^{\alpha(x)}| v| ^{\beta(x)} dx\nonumber\\
	&\leq\|c\|_{L^{\infty}(\Omega)} \DD\int_{\Omega}\Big\{ \frac{\alpha(x)}{\alpha(x)+\beta(x)}|u|^{\alpha(x)+\beta(x)}+\frac{\beta(x)}{\alpha(x)+\beta(x)}|v|^{\alpha(x)+\beta(x)}\Big\}dx\nonumber\\
	&\leq\|c\|_{L^{\infty}(\Omega)}\Big[\Big\{ \|u\|^{\alpha^++\beta^+}_{L^{\alpha(x)+\beta(x)}(\Omega)} + \|u\|^{\alpha^-+\beta^-}_{L^{\alpha(x)+\beta(x)}(\Omega)}\Big\}\nonumber\\&~~~~~~~~~~~~~~~~~~~~~~~~~~~~ +\Big\{ \|v\|^{\alpha^++\beta^+}_{L^{\alpha(x)+\beta(x)}(\Omega)} + \|v\|^{\alpha^-+\beta^-}_{L^{\alpha(x)+\beta(x)}(\Omega)}\Big\}\Big]\nonumber\\
	&\leq K_2 \Big[\Big\{ \|u\|^{\alpha^++\beta^+}_{X_0} + \|u\|^{\alpha^-+\beta^-}_{X_0}\Big\}+ \Big\{ \|v\|^{\alpha^++\beta^+}_{X_0} + \|v\|^{\alpha^-+\beta^-}_{X_0}\Big\}\Big]\nonumber\\
	&\leq C_2\max \Big\{\|u\|_{X_0}^{\alpha^-+\beta^-},~  \|u\|_{X_0}^{\alpha^++\beta^+}, ~ \|v\|_{X_0}^{\alpha^-+\beta^-},~ \|v\|_{X_0}^{\alpha^++\beta^+}\Big\}\nonumber\\
	&= C_2\max \Big\{\max \Big\{\|u\|_{X_0}^{\alpha^-+\beta^-}, \|v\|_{X_0}^{\alpha^-+\beta^-}\Big\},~ \max \Big\{\|u\|_{X_0}^{\alpha^++\beta^+},~ \|v\|_{X_0}^{\alpha^++\beta^+}\Big\}\Big\}\nonumber\\
	&= C_2\max \Big\{\|(u,v)\|^{\alpha^-+\beta^-},~ \|(u,v)\|^{\alpha^++\beta^+}\Big\},
	\end{align*}}
	where $K_2=\|c\|_{L^{\infty}(\Omega)}.\max\Big\{\Big(C( N,s,p,\alpha,\beta,\Omega)\Big)^{\alpha^-+\beta^-}, \Big(C( N,s,p,\alpha,\beta,\Omega)\Big)^{\alpha^++\beta^+}\Big\}$ and\\ $C_2=4K_2+1.$
\end{proof} 
Now we have the following lemma.
\begin{lemma}\label {lem1}
	Let $(u^*,v^*)\in \mathscr N_{\la,\mu}^+$ ( or $\in \mathscr N_{\la,\mu}^-)$ be a local minimizer of $J_{\la,\mu}$ on $\mathscr N_{\la,\mu}^+$ (or on $\mathscr N_{\la,\mu}^-)$. Then $(u^*,v^*)$ is 
	a critical point of $J_{\la,\mu}.$
\end{lemma}
\begin{proof}
	First assume that $(u^*,v^*)\in \mathscr N_{\la,\mu}^+$ is a local minimizer of $J_{\la,\mu}$ on $\mathscr N_{\la,\mu}^+$.
	Let $I_{\la,\mu}(u,v) = \langle J'_{\la,\mu}(u,v),(u,v)\rangle$. 
	Note that for $(u,v)\in E\setminus\{0\}$ with $I_{\la,\mu}(u,v)=0$, we have $\varphi''_{u,v}(1)>0$ if and only if 
	$\langle I'_{\la,\mu}(u,v),(u,v)\rangle>0$.
	Now as $(u^*,v^*)$ is a local minimizer of $J_{\la,\mu}$  on $\mathscr N_{\la,\mu}^+$, 
	using Lagrange's multiplier theorem we get a real number 
	$\tau$ such that $$J'_{\la,\mu}(u^*,v^*) =\tau I'_{\la,\mu}(u^*,v^*).$$ 
	Therefore $$0=\langle J'_{\la,\mu}(u^*,v^*),(u^*,v^*)\rangle=\tau \langle I'_{\la,\mu}(u^*,v^*),(u^*,v^*)\rangle=\tau \phi''_{(u^*,v^*)}(1) .$$
	As $(u^*,v^*)\in \mathscr N_{\la,\mu}^+ $, we get $\phi''_{(u^*,v^*)}(1)>0$ and hence $\tau= 0.$ This completes the proof.
	Similarly we can prove the result  when $(u^*,v^*)\in \mathscr N_{\la,\mu}^-$ is
	a local minimizer of $J_{\la,\mu}$ on $\mathscr N_{\la,\mu}^-$.
\end{proof}
\begin{lemma}\label{lem5} 
	There exists $\delta>0,$ given by $$\delta=\frac{1}{C_1}\bigg( \frac{\alpha^-+\beta^--p^+}
	{\alpha^-+\beta^--q^-}\bigg)\bigg( \frac{p^--q^+}{C_2(\alpha^++\beta^+-q^+)}\bigg)^{\DD\frac{p^+-q^-}{\alpha^-+\beta^--p^+}}$$ 
	such that for any pair of $(\la,\mu)\in \RR^+\times \RR^+$ with $\la+\mu<\delta,$ we have $\mathscr N_{\la,\mu}^0=\emptyset,$ 
	where the positive constants $C_1,C_2$ are as given  in Lemma \ref{lem4}.
\end{lemma}
\begin{proof}
	We prove this lemma by contradiction. Let us assume that  there exist $\la,\mu>0$ with $\la+\mu<\delta$ such that
	$\mathscr N_{\la,\mu}^0\not=\emptyset.$ Hence there is  $(u,v)\in \mathscr N_{\la,\mu}^0. $ 
	Now if $\|(u,v)\|<1,$ 
	then from \eqref{1.0} using \eqref{00} and Lemma \ref{lem4} $(i), (ii)$, we obtain 
	\begin{align*}
	0=\varphi''_{(u,v)}(1)
	&\leq p^+ P(u,v)-q^-Q(u,v)-(\alpha^-+\beta^-)R(u,v)\\
	&= (p^+-(\alpha^-+\beta^-)) P(u,v)
	+ (\alpha^-+\beta^--q^-)Q(u,v) \\
	&\leq (p^+-(\alpha^-+\beta^-)) \|(u,v)\|^{p^+} + (\alpha^-+\beta^--q^-) C_1 (\la+\mu) \|(u,v)\|^{q^-}.
	\end{align*}
	This implies
	\begin{align}\label{1.8}
	\|(u,v)\|^{p^+-q^-}\leq \frac{(\alpha^-+\beta^--q^-)}{(\alpha^-+\beta^--p^+)} C_1 (\la+\mu).
	\end{align}
	Again from \eqref{1.0}, using \eqref{00} and Lemma \ref{lem4} $(i), (iii)$, we deduce that
	\begin{align*}
	0=\varphi''_{(u,v)}(1)
	&\geq p^- P(u,v) -q^+Q(u,v)-(\alpha^++\beta^+)R(u,v)\nonumber\\
	&=(p^--q^+) P(u,v) 
	- (\alpha^++\beta^+-q^+)R(u,v)\nonumber\\
	&\geq (p^--q^+) \|(u,v)\|^{p^+} - (\alpha^++\beta^+-q^+) C_2~ \|(u,v)\|^{\alpha^-+\beta^-}.
	\end{align*}
	This gives
	\begin{align}\label{1.9}
	1\geq\|(u,v)\|^{\alpha^-+\beta^--p^+}\geq \frac{(p^--q^+)}{C_2(\alpha^++\beta^+-q^+)} .
	\end{align}
	Combining \eqref{1.8} and \eqref{1.9}, we get 
	$$\la+\mu\geq \frac{1}{C_1}\bigg( \frac{\alpha^-+\beta^--p^+}{\alpha^-+\beta^--q^-}\bigg)
	\bigg( \frac{p^--q^+}{C_2(\alpha^++\beta^+-q^+)}\bigg)^{\DD\frac{p^+-q^-}{\alpha^-+\beta^--p^+}},$$
	which is a contradiction.	
	Now, if $ \|(u,v)\|>1,$ again using \eqref{1.0}, \eqref{00} and Lemma \ref{lem4} $(i), (ii)$, we obtain
	$$0=\varphi''_{u,v}(1)\leq(p^+-(\alpha^-+\beta^-)) \|(u,v)\|^{p^-} + (\alpha^-+\beta^--q^-) C_1 (\la+\mu) \|(u,v)\|^{q^+} ,
	$$
	that is, 
	\begin{align}\label{1.10}
	\|(u,v)\|^{p^--q^+}\leq \frac{(\alpha^-+\beta^--q^-)}{(\alpha^-+\beta^--p^+)} C_1 (\la+\mu).
	\end{align}
	On the other hand from \eqref{1.0}, \eqref{00} and Lemma \ref{lem4} $(i), (iii),$ we find 
	$$0=\varphi''_{u,v}(1)\geq (p^--q^+) \|(u,v)\|^{p^-} - (\alpha^++\beta^+-q^+) C_2~ \|(u,v)\|^{\alpha^++\beta^+} ,$$
	that is,
	\begin{align}\label{1.11}
	\|(u,v)\|^{\alpha^++\beta^+-p^-}\geq \frac{(p^--q^+)}{C_2(\alpha^++\beta^+-q^+)} .
	\end{align} 
	Thus combining \eqref{1.10} and \eqref{1.11},
	\begin{align}\label{55}
	\la+\mu\geq \frac{1}{C_1}\bigg( \frac{\alpha^-+\beta^--p^+}{\alpha^-+\beta^--q^-}\bigg)\bigg( \frac{p^--q^+}{C_2(\alpha^++\beta^+-q^+)}\bigg)^{\DD\frac{p^--q^+}{\alpha^++\beta^+-p^-}}.
	\end{align}
	Since $0<\DD\bigg( \frac{p^--q^+}{C_2(\alpha^++\beta^+-q^+)}\bigg)<1$ and also $\DD\frac{p^--q^+}{\alpha^++\beta^+-p^-}<\frac{p^+-q^-}{\alpha^-+\beta^--p^+},$
	from \eqref{55} we finally get 
{\small	$$\la+\mu\geq \frac{1}{C_1}\bigg( \frac{\alpha^-+\beta^--p^+}{\alpha^-+\beta^--q^-}\bigg)\bigg( \frac{p^--q^+}{C_2(\alpha^++\beta^+-q^+)}\bigg)^{\DD\frac{p^+-q^-}{\alpha^-+\beta^--p^+}},$$}
	which is a contradiction.
	Hence the lemma is proved. 
\end{proof}
In the next result, we discuss the behavior of the functional $J_{\la,\mu}$ on   $\mathscr N_{\la,\mu}.$
\begin{lemma}\label{lem6}
	For $\la+\mu<\delta$, $J_{\la,\mu}$ is coercive and bounded below on $\mathscr N_{\la,\mu}.$
\end{lemma}
\begin{proof}
	Let $(u,v)\in \mathscr N_{\la,\mu}.$ Then for $\|(u,v)\|>1, $ from \eqref{energy} and \eqref{1.0} and Lemma \ref{lem4} $(ii)$, we obtain
	{\small\begin{align}\label{1.12}
	J_{\la,\mu}(u,v)	&\geq \frac{1}{p^+}P(u,v)- \frac{1}{q^-}Q(u,v)- \frac{1}{\alpha^-+\beta^-}R(u,v)\nonumber\\
	&=\Big(\frac{1}{p^+}-\frac{1}{\alpha^-+\beta^-}\Big)P(u,v)
	-\Big(\frac{1}{q^-}-\frac{1}{\alpha^-+\beta^-}\Big)Q(u,v) \nonumber\\
	&\geq \Big(\frac{1}{p^+}-\frac{1}{\alpha^-+\beta^-}\Big)\|(u,v)\|^{p^-}-C_1(\la+\mu)\Big(\frac{1}{q^-}-\frac{1}{\alpha^-+\beta^-}\Big)\|(u,v)\|^{q^+}.
	\end{align}}
	Since from $(A1)$, we have $1<q^-\leq q^+<p^-\leq p^+<\alpha^-+\beta^-,$  
	we conclude from \eqref{1.12}  that $J_{\la,\mu}(u,v)\ra +\infty$ as $\|(u,v)\|\ra+\infty.$ Therefore $J_{\la,\mu}$
	is coercive and bounded below on $\mathscr N_{\la,\mu}.$
\end{proof}
\begin{lemma}\label{lem*}We have the following results:
	\begin{itemize}
		\item[$(i).$] If $(u,v)\in\mathscr N^+_{\la,\mu},$ then $Q(u,v)>0.$
		\item[$(ii).$]If $(u,v)\in\mathscr N^-_{\la,\mu},$ then $R(u,v)>0.$
	\end{itemize}
\end{lemma}
\begin{proof}
	$(i).$ Since $(u,v)\in\mathscr N^+_{\la,\mu},$ we have $\phi''_{(u,v)}(1)>0.$
	Thus using \eqref{1.0} and \eqref{00}, we obtain
	\begin{align*}
	0<\varphi''_{(u,v)}(1)&\leq p^+ P(u,v)-q^-Q(u,v) -(\alpha^-+\beta^-)R(u,v)\\
	&=  \{p^+-(\alpha^-+\beta^-)\} P(u,v)
	+(\alpha^-+\beta^--q^-)Q(u,v).
	\end{align*}
	This implies that $Q(u,v)
	\geq\DD\frac{(\alpha^-+\beta^--p^+)}{(\alpha^-+\beta^--q^-)}
	P(u,v)>0.$\\
	\noindent$(ii).$ Since $(u,v)\in\mathscr N^-_{\la,\mu},$ we have $\phi''_{(u,v)}(1)<0.$
	Thus using \eqref{1.0} and \eqref{00}, we obtain
	\begin{align*}
	0>\varphi''_{(u,v)}(1)&\geq p^- P(u,v)-q^+Q(u,v)-(\alpha^++\beta^+)R(u,v)\\
	&=( p^--q^+) P(u,v)-(\alpha^++\beta^+-q^+)R(u,v), 
	\end{align*}
	that is, 
	$R(u,v)
	\geq\DD\frac{(\alpha^++\beta^+-p^-)}{(\alpha^++\beta^+-q^+)}P(u,v)>0.$
\end{proof}
\begin{remark} From Lemma \ref{lem5} and Lemma \ref{lem6}, we conclude that for any pair of parameters $(\la,\mu)\in \RR^+\times\RR^+$ 
	with $\la+\mu<\delta,$ $\mathscr N_{\la,\mu}=\mathscr N^-_{\la,\mu}\cup \mathscr N^+_{\la,\mu}$ and $J_{\la,\mu}$
	is coercive and bounded below on $\mathscr N^-_{\la,\mu}$ and $\mathscr N^+_{\la,\mu}.$ Therefore we can define
	$$\theta_{\la,\mu}=\DD\inf_{(u,v)\in \mathscr N_{\la,\mu}}J_{\la,\mu}(u,v);~ \theta^+_{\la,\mu}=
	\DD\inf_{(u,v)\in \mathscr N^+_{\la,\mu}}J_{\la,\mu}(u,v);~\theta^-_{\la,\mu}=\DD\inf_{(u,v)\in \mathscr N^-_{\la,\mu}}J_{\la,\mu}(u,v).$$
\end{remark}
Now we establish some important properties of $\mathscr N^+_{\la,\mu}$ and $\mathscr N^-_{\la,\mu}$ in the next two lemmas, respectively.
\begin{lemma}\label{lem7}
	If $\la+\mu<\delta,$ then $\theta_{\la,\mu}\leq \theta^+_{\la,\mu}<0.$
\end{lemma}
\begin{proof}
	Let $(u,v)\in  \mathscr N^+_{\la,\mu}.$ Then $\varphi''_{u,v}(1)>0.$
	Now combining  \eqref{1.0} and \eqref{00}, we obtain
	{\small\begin{align*}
	0<\varphi''_{u,v}(1)&<p^+P(u,v)- q^-Q(u,v) - (\alpha^-+\beta^-)R(u,v)\\
	&=(p^+-q^-)P(u,v)
	- (\alpha^-+\beta^--q^-)R(u,v),
	\end{align*}}
	that is,
	{\small\begin{align}\label{1.13}
	R(u,v)
	&<\frac{(p^+-q^-)}{(\alpha^-+\beta^--q^-)}P(u,v).
	\end{align}}
	Using \eqref{1.0} and \eqref{1.13}, from \eqref{energy} we deduce
	{\small\begin{align}\label{1.14}
		J_{\la,\mu}(u,v)&\leq \frac{1}{p^-}P(u,v)
		-\frac{1}{q^+}Q(u,v)-\frac{1}{\alpha^++\beta^+}R(u,v)\nonumber\\
		&= \Big(\frac{1}{p^-}-\frac{1}{q^+}\Big)P(u,v)
		+ \Big(\frac{1}{q^+}-\frac{1}{\alpha^++\beta^+}\Big)R(u,v)\nonumber\\
		&\leq \bigg\{ \Big(\frac{1}{p^-}-\frac{1}{q^+}\Big)+ \Big(\frac{1}{q^+}-\frac{1}{\alpha^++\beta^+}\Big) \frac{(p^+-q^-)}{(\alpha^-+\beta^--q^-)}  \bigg\}P(u,v)\nonumber\\
		&=\Bigg\{\frac{(q^+-p^-)(\alpha^++\beta^+)+p^-(\alpha^++\beta^+-q^+)\DD\frac{(p^+-q^-)}{(\alpha^-+\beta^--q^-)}}{p^- q^+ (\alpha^++\beta^+)}\Bigg\}P(u,v).
		\end{align}}
	From $(A2), $ we have ${(q^+-p^-)(\alpha^++\beta^+)+p^-(\alpha^++\beta^+-q^+)\DD\frac{(p^+-q^-)}{(\alpha^-+\beta^--q^-)}}<0.$ 
	Hence \eqref{1.14} implies $J_{\la\mu}(u,v)<0.$ Therefore from the definition of  $\theta_{\la,\mu},$  and 
	$ \theta^+_{\la,\mu},$ it follows that $\theta_{\la,\mu}\leq \theta^+_{\la,\mu}<0.$
\end{proof}
\begin{lemma}\label{lem8}
	If $\la+\mu<\Big(\frac{q^-}{p^+}\Big)\delta,$ then $ \theta^-_{\la,\mu}>K,$ where $K=K(N,s,p,q,\alpha,\beta,a,b,\la,\mu,\Omega)$ is some positive constant.
\end{lemma}
\begin{proof}
	Let $(u,v)\in  \mathscr N^-_{\la,\mu}.$ Then $\varphi''_{u,v}(1)<0. $ Therefore from \eqref{1.9} and \eqref{1.10}, we obtain{\small\begin{align}
	\label{000}
	\begin{cases}
	\|(u,v)\|\geq \bigg\{\frac{(p^--q^+)}{C_2(\alpha^++\beta^+-q^+)}\bigg\}^{1/(\alpha^-+\beta^--p^+)}
	,& \|(u,v)\|<1 \\ 
	\|(u,v)\|\geq \bigg\{\frac{(p^--q^+)}{C_2(\alpha^++\beta^+-q^+)}\bigg\}^{1/(\alpha^++\beta^+-p^-)}
	,& \|(u,v)\|>1.
	\end{cases}\end{align}}
	Now if $\|(u,v)\|<1,$ then plugging \eqref{1.0}  into \eqref{energy} and using Lemma \ref{lem4} $(i),(ii)$ and  \eqref{000}, we deduce
	{\small	\begin{align}\label{1.16}
		J_{\la,\mu}(u,v)&\geq \frac{1}{p^+}P(u,v)
		-\frac{1}{q^-}Q(u,v)-\frac{1}{\alpha^-+\beta^-}R(u,v)\nonumber\\
		&= \Big(\frac{1}{p^+}-\frac{1}{\alpha^-+\beta^-}\Big)P(u,v)
		- \Big(\frac{1}{q^-}-\frac{1}{\alpha^-+\beta^-}\Big)Q(u,v)\nonumber\\
		&\geq \Big(\frac{1}{p^+}-\frac{1}{\alpha^-+\beta^-}\Big) \|(u,v)\|^{p^+}-\Big(\frac{1}{q^-}-\frac{1}{\alpha^-+\beta^-}\Big)C_1(\la+\mu)\|(u,v)\|^{q^-}\nonumber\\
		&= \|(u,v)\|^{q^-} \bigg[\Big(\frac{1}{p^+}-\frac{1}{\alpha^-+\beta^-}\Big) \|(u,v)\|^{p^+-q^-}-\Big(\frac{1}{q^-}-\frac{1}{\alpha^-+\beta^-}\Big)C_1(\la+\mu)\bigg]\nonumber\\&\geq\bigg\{\frac{(p^--q^+)}{C_2(\alpha^++\beta^+-q^+)}\bigg\}^{\DD\frac{q^-}{(\alpha^-+\beta^--p^+)}}\bigg[\Big(\frac{1}{p^+}-\frac{1}{\alpha^-+\beta^-}\Big)\nonumber\\ &~~~~~~~~~\bigg\{\frac{(p^--q^+)}{C_2(\alpha^++\beta^+-q^+)}\bigg\}^{\DD\frac{(p^+-q^-)}{(\alpha^-+\beta^--p^+)}}-\Big(\frac{1}{q^-}-\frac{1}{\alpha^-+\beta^-}\Big)C_1(\la+\mu)\bigg]=d_1
		\end{align}}
	Next, if{\small $$\la+\mu< \Big(\frac{ q^-}{p^+}\Big)\delta= \Big(\frac{ q^-}{p^+}\Big)\frac{1}{C_1}\bigg( \frac{\alpha^-+\beta^--p^+}{\alpha^-+\beta^--q^-}\bigg)\bigg\{ \frac{(p^--q^+)}{C_2(\alpha^++\beta^+-q^+)}\bigg\}^{\DD\frac{(p^+-q^-)}{(\alpha^-+\beta^--p^+)}},$$} then {\small$$\la+\mu< \frac{{\alpha^-+\beta^-}-p^+}{p^+ (\alpha^-+\beta^-)}\bigg\{\frac{(p^--q^+)}{C_2(\alpha^++\beta^+-q^+)}\bigg\}^{\DD\frac{(p^+-q^-)}
		{(\alpha^-+\beta^--p^+)}}\frac{(\alpha^-+\beta^-)q^-}{\alpha^-+\beta^--q^-}.\frac{1}{C_1},$$} that is,
	{\small$$\Big(\frac{1}{p^+}-\frac{1}{\alpha^-+\beta^-}\Big) \bigg\{\frac{(p^--q^+)}{C_2(\alpha^++\beta^+-q^+)}\bigg\}^{\DD\frac{(p^+-q^-)}{(\alpha^-+\beta^--p^+)}}-\Big(\frac{1}{q^-}-\frac{1}{\alpha^-+\beta^-}\Big)C_1(\la+\mu)>0,$$} thus from \eqref{1.16}, we get $d_1>0.$ 
	\par Similarly for $\|(u,v)\|>1,$ again plugging \eqref{1.0}  in \eqref{energy} and using Lemma \ref{lem4} $(i),(ii)$ and  \eqref{000}, we obtain
	{\small\begin{align}\label{1.17}
		J_{\la,\mu}(u,v)&\geq \frac{1}{p^+}P(u,v)
		-\frac{1}{q^-}Q(u,v)
		-\frac{1}{\alpha^-+\beta^-}R(u,v)\nonumber\\
		&= \Big(\frac{1}{p^+}-\frac{1}{\alpha^-+\beta^-}\Big)P(u,v)
		- \Big(\frac{1}{q^-}-\frac{1}{\alpha^-+\beta^-}\Big)  Q(u,v)\nonumber\\
		&\geq \Big(\frac{1}{p^+}-\frac{1}{\alpha^-+\beta^-}\Big) \|(u,v)\|^{p^-}-\Big(\frac{1}{q^-}-\frac{1}{\alpha^-+\beta^-}\Big)C_1(\la+\mu)\|(u,v)\|^{q^+}\nonumber\\
		&= \|(u,v)\|^{q^+} \bigg\{\Big(\frac{1}{p^+}-\frac{1}{\alpha^-+\beta^-}\Big) \|(u,v)\|^{p^--q^+}-\Big(\frac{1}{q^-}-\frac{1}{\alpha^-+\beta^-}\Big)C_1(\la+\mu)\bigg\}\nonumber\\&\geq \bigg\{\frac{(p^--q^+)}{C_2(\alpha^++\beta^+-q^+)}\bigg\}^{\DD\frac{q^+}{(\alpha^++\beta^+-p^-)}}\bigg[ \Big(\frac{1}{p^+}-\frac{1}{\alpha^-+\beta^-}\Big)\nonumber\\ &~~~\bigg\{\frac{(p^--q^+)}{C_2(\alpha^++\beta^+-q^+)}\bigg\}^{\DD\frac{(p^--q^+)}{(\alpha^++\beta^+-p^-)}}-\Big(\frac{1}{q^-}-\frac{1}{\alpha^-+\beta^-}\Big)C_1(\la+\mu)\bigg]
		\end{align}}
	Combining the facts that {\small$\DD\frac{(p^--q^+)}{(\alpha^++\beta^+-p^-)}<\DD\frac{(p^+-q^-)}{(\alpha^-+\beta^--p^+)}$} and {\small $\bigg\{ \frac{(p^--q^+)}{C_2(\alpha^++\beta^+-q^+)}\bigg\}<1,$ } and plugging \eqref{1.16} into \eqref{1.17}, we can deduce 
	{\small	\begin{align*}J_{\la,\mu}(u,v)&\geq
		\bigg\{\frac{(p^--q^+)}{C_2(\alpha^++\beta^+-q^+)}\bigg\}^{\DD\frac{q^+}{(\alpha^-+\beta^--p^+)}}\bigg[\Big(\frac{1}{p^+}-\frac{1}{\alpha^-+\beta^-}\Big)\\ &~~~~~~~~~~~~\bigg\{\frac{(p^--q^+)}{C_2(\alpha^++\beta^+-q^+)}\bigg\}^{\DD\frac{(p^+-q^-)}{(\alpha^-+\beta^--p^+)}}-\Big(\frac{1}{q^-}-\frac{1}{\alpha^-+\beta^-}\Big)C_1(\la+\mu)\bigg]\end{align*}
		\begin{align*}
		&\geq \bigg\{\frac{(p^--q^+)}{C_2(\alpha^++\beta^+-q^+)}\bigg\}^{\DD\frac{(q^+-q^-)}{(\alpha^-+\beta^--p^+)}} d_1 = d_2>0.
		\end{align*} }
	Finally by choosing $K=\min\{d_1,d_2\}>0,$ the lemma holds.	 
\end{proof}
Next lemma gives the nature of the  map $\varphi_{u,v}$ .
We refer \cite{Wu,chen-deng} for the same result in case of local and nonlocal $p-$Laplacian and \cite{ alves,Fan2} for variable exponent Laplacian.
\begin{lemma}\label{lem3}
	For $(u,v)\in E\setminus\{(0,0)\}$, there exists $\delta'>0$ such that for all $\la+\mu<\delta',$  we have the followings:
	\begin{itemize}
		\item [$(i).$] If $Q(u,v)=0,$ then there exists unique $t^-=t^-(u,v)$ such that $(t^-u, t^-v)\in \mathscr N^-_{\la,\mu}$ and
		$J_{\la,\mu}(t^-u,t^-v)=\DD\sup_{t\geq0}J_{\la,\mu}(tu,tv).$
		\item[$(ii).$] If  $Q(u,v)>0,$ then there exist $t^*>0$ and unique positive numbers $t^+=t^+(u,v)<t^-=t^-(u,v)$ 
		such that $(t^-u, t^-v)\in \mathscr N^-_{\la,\mu}$, $(t^+u, t^+v)\in \mathscr N^+_{\la,\mu}$ and
		$$ J_{\la,\mu}(t^+u,t^+v)=\inf_{0\leq t\leq t^*}J_{\la,\mu}(tu,tv);~ J_{\la,\mu}(t^-u,t^-v)=\sup_{t\geq 0}J_{\la,\mu}(tu,tv).$$
	\end{itemize} 
\end{lemma}
\begin{proof}
	$(i).$	Using the given assumption,
	for $0<t<1$ sufficiently small,
	{\small\begin{align*}\varphi_{u,v}(t)&>\frac{t^{p^+}}{p^+}P(u,v) -\frac{t^{\alpha^++\beta^+}}{\alpha^++\beta^+}R(u,v)>0\end{align*}} 
	and for $t>1$ sufficiently large 
	{\small\begin{align*}\varphi_{u,v}(t)&<\frac{t^{p^+}}{p^-}P(u,v) -\frac{t^{\alpha^++\beta^+}}{\alpha^-+\beta^-}R(u,v) <0.\end{align*}}
	Hence $\varphi_{u,v}$ achieves its maximum at some point $t^-(u,v)$ on $[0,\infty).$ 
	Thus we have $\varphi'_{u,v}(t^-)$ $=\langle J'_{\la,\mu}(t^-u,t^-v), (u,v)\rangle$ $=0.$ Set $(t^-u,t^-v)=(\overline{u},\overline{v}).$
	Then $\langle J'_{\la,\mu}(\overline{u},\overline{v}), (\overline{u},\overline{v})\rangle$ $=0,$ 
	which implies $(\overline{u},\overline{v})\in\mathscr{N}_{\la,\mu}.$ Therefore inserting from \eqref{1.0}, we get 
	\begin{align}\label{6.1}
	P(\overline{u},\overline{v})=R(\overline{u},\overline{v}).
	\end{align}
	Now we define the function $\Theta_{\overline{u},\overline{v}}:[0,\infty)\ra\RR$ as 
	$\Theta_{\overline{u},\overline{v}}(t)=J_{\la,\mu}(t\overline{u},t\overline{v}).$
	We know that $\Theta_{\overline{u},\overline{v}}(1)=J_{\la,\mu}(\overline{u},\overline{v})=\DD\max_{t\in[0,\infty)} \Theta_{\overline{u},\overline{v}}(t)$ and $\Theta'_{\overline{u},\overline{v}}(1)=\langle J'_{\la,\mu}(\overline{u},\overline{v}), (\overline{u},\overline{v})\rangle=0.$ For $t>1,$ by \eqref{6.1} we obtain
	\begin{align*}
	\Theta'_{\overline{u},\overline{v}}(t)&=\langle J'_{\la,\mu}(t\overline{u},t\overline{v}), (\overline{u},\overline{v})\rangle\\
	&\leq t^{p^+-1}P(\overline{u},\overline{v})-t^{{\alpha^-+\beta^-}-1}R(\overline{u},\overline{v})<0,
	\end{align*}
	and on the other hand for $t\in (0,1)$ again using \eqref{6.1}, we obtain
	\begin{align*}
	\Theta'_{\overline{u},\overline{v}}(t)&=\langle J'_{\la,\mu}(t\overline{u},t\overline{v}), (\overline{u},\overline{v})\rangle\\
	&\geq t^{p^+-1}P(\overline{u},\overline{v})-t^{{\alpha^-+\beta^-}-1}R(\overline{u},\overline{v})>0.
	\end{align*}
	This shows that the point $t^-$ is unique. Hence the result follows.\\
	$(ii).$ To prove this lemma, first we set
	\begin{align*}f_1(t)&:=\DD\int_{\RR^N\times\RR^N}t^{p(x,y)}\bigg\{\frac{|
		u(x)-u(y)|^{p(x,y)}}{|
		x-y|^{N+sp(x,y)}}+\frac{|
		v(x)-v(y)|^{p(x,y)}}{|
		x-y|^{N+sp(x,y)}}\bigg\}dxdy;
	\\f_2(t)&:=\DD\int_{\Omega}t^{q(x)}\Big(\la a(x)| u|^{q(x)}+
	\mu b(x)| v|^{q(x)}\Big) dx;\\ 
	f_3(t)&:=\DD\int_{\Omega}t^{\alpha(x)+\beta(x)}c(x)| u|^{\alpha(x)}| v| ^{\beta(x)} dx.\end{align*} 
	Then $f_i$'s are continuous and strictly  increasing functions with $f_i(0)=0$ for $i=1,2,3$.
	Also we  have the following observations.
	{\begin{itemize}
		\item [(I).] $\DD\lim_{t\ra0^+} \frac{f_3(t)}{f_1(t)}=0.$
		\item [(II).]$\DD\lim_{t\ra+\infty} f_2(t)=+\infty.$
		\item [(III).]$\DD\lim_{t\ra+\infty} \frac{(f_1-f_3)(t)}{f_2(t)}=0.$
		\item [(IV).]$f_1-f_3$ has unique point of maximum
		, say $ {t}_{max}$ and $(f_1-f_3)(t)\ra-\infty$  as $t\ra +\infty.$ 
		\item [(V).]  There exists $\tilde{t}\in (0, {t}_{max})$ such that $\DD\frac{f_1-f_3}{f_2}$ is strictly increasing on $(0,\tilde{t}).$
	\end{itemize}}
	From (I), we note that $(f_1-f_3)(t)>0$ for $t\ra 0^+$ sufficiently small. Hence using (V) and intermediate value theorem, 
	we have that for each choice of the pair  $(\la,\mu)\in\RR^+\times\RR^+$ with $f_2(\tilde t)<(f_1-f_3)(\tilde t)$, 
	there exists a unique $t^+=t^+({\la,\mu})\in (0,\tilde{t})$ such that
	{\small\begin{align}\label{5} 
	\frac{(f_1-f_3)(t^+)}{f_2(t^+)}=1.
	\end{align}}
	Since $\frac{(f_1-f_3)}{f_2}$ is strictly monotone increasing in $(t^+,\tilde{t})$,  
	from \eqref{5}, we get 
{\small\begin{align*}
	\DD1=\frac{(f_1-f_3)(t^+)}{f_2(t^+)}<\frac{(f_1-f_3)(t)}{f_2(t)}\;\text{ for all }t\in (t^+,\tilde{t}), 
	\end{align*}}that is,
	\begin{align}\label{5.0}
	f_2(t)<(f_1-f_3)(t) ~\text{ for all ~} t\in (t^+,\tilde{t}). 
	\end{align} 
	Now we can fix $(\la^*,\mu^*)\in\RR^+\times\RR^+$ such that
	for all $\la\in(0,\la^*),\mu\in(0,\mu^*)$, taking into account \eqref{5.0}, we have 
	\begin{align}\label{5.2}f_2(t)<(f_1-f_3)({t})\text{ for all }t\in (t^+,t_{\max}).\end{align}
	Since $f_1-f_3$ is strictly decreasing in $({t}_{\max},\infty)$ and  $f_2$ is monotonically increasing in $(0,\infty)$, using (II),
	it follows from \eqref{5.2} that there exists a unique positive real number $t^{-}> {t}_{max}$ such that 
	\begin{align}\label{5.3}
	f_2(t^-)=(f_1-f_3)(t^-) \text{~ for~all ~$(\la,\mu)\in(0,\la^*)\times(0,\mu^*).$} 
	\end{align}
	Hence from \eqref{5} and \eqref {5.3}, it follows that the function
	$\varphi'_{u,v}(t)=f_1-f_2-f_3$ has exactly two nontrivial zeros, $t^+<  ~t^-,$ that is, $t^+$ and $t^-$ are critical points of
	$\varphi_{u,v}(t)$.
	For $\delta':=\la^*+\mu^*$, we can choose $\la^*,\mu^*>0$ sufficiently small such that $\delta'<\delta,$ where $\delta$ is as given in Lemma \ref{lem5}. Then as
	$\varphi_{u,v}(0)=0$ and $\varphi_{u,v}(t)<0$ for $t\ra0^+$ sufficiently small,
	we get $\varphi'_{u,v}(t)<0$ for all $t\in(0,t^+)$ and $\varphi'_{u,v}(t)>0$ for all $t\in(t^+,t_{max})$  and
	$\varphi'_{u,v}(t^+)=0.$ Now as from Lemma \ref{lem5}, we have $\mathcal{N}_{\la,\mu}^0=\emptyset$, we can conclude that  $\varphi_{u,v}$ attains a local minimum  at $t^+$ and consequently 
	$\varphi''_{u,v}(t^+)>0.$ Hence $(t^+u, t^+v)\in \mathscr{N}^+.$
	\par Similarly as  we have $\varphi'_{u,v}(t)>0$ for all $t\in[t_{max},t^-)$, $\varphi'_{u,v}(t)<0$ for all 
	$t>t^-$ and $\varphi'_{u,v}(t^-)=0,$ from Lemma \ref{lem5} using the fact  $\mathcal{N}_{\la,\mu}^0=\emptyset,$ 
	it follows that  $t^-$ is the point of global maximum for $\varphi_{u,v}$ and 
	consequently $\varphi''_{u,v}(t^-)<0.$ Hence $(t^-u,t^-v)\in \mathscr{N}^-. $ 
	Now combining Lemma \ref{lem7} and Lemma \ref{lem8}, we obtain $\varphi_{u,v}(t^+)<0$ and  $\varphi_{u,v}(t^-)>0.$
	Also from the above discussion, we get that $\varphi_{u,v}$ is strictly increasing  on $[t^+,t^-],$ and strictly decreasing for all $t> t^-$ with
	$\varphi_{u,v}(t)\ra-\infty$ as $t\ra+\infty.$ Thus there exists a unique $t^*\in (t^+,t^-) $  such that $\varphi_{u,v}(t^*)=0$.
	Therefore{\small $$J_{\la,\mu}(t^+u,t^+v)=\varphi_{u,v}(t^+)=\inf_{0\leq t\leq t^*}\phi_{u,v}(t)=\inf_{0\leq t\leq t^*}J_{\la,\mu}(tu,tv)$$}
	and
	{\small$$J_{\la,\mu}(t^-u,t^-v)=\varphi_{u,v}(t^-)=\DD\sup_{t\geq 0}\phi_{u,v}(t)=\sup_{t\geq 0}J_{\la,\mu}(tu,tv).$$}
	This completes the lemma.	
\end{proof}
\section{ Existence of multiple solutions}\label{manifold}
In this section we give the proof of the existence of  at least two distinct non-trivial and non-negative weak solutions for the problem \ref{mainprob}. The next two propositions  ensure  the existence of  minimizers for the functional $J_{\la,\mu}$ in $\mathscr N^+_{\la,\mu}$ and  $\mathscr N^-_{\la,\mu}$, respectively, which serve as weak solutions to problem \eqref{mainprob}.
We set $\delta_0:=\min\Big\{\Big(\frac{q^-}{p^+}\Big)\delta, \delta'\Big\},$ where $\Big(\frac{q^-}{p^+}\Big)\delta$ and $\delta'$ are as given in  
Lemma \ref{lem8} and Lemma \ref{lem3}, respectively.
\begin{proposition}\label{lem9}
	For $\la+\mu<\delta_0,$   the functional $	J_{\la,\mu}$ has a  minimizer $(u_0,v_0)$ in $\mathscr N^+_{\la,\mu},$ which satisfies 
	the followings:
	\begin{itemize}
		\item [$(i).$]  $	J_{\la,\mu}(u_0,v_0)=\theta_{\la,\mu}^+<0;$ 
		\item[ $(ii).$] $(u_0,v_0)$ is a  solution of the problem \eqref{mainprob} 
	\end{itemize}	
\end{proposition}
\begin{proof} $(i)$
	Since $	J_{\la,\mu}$ is bounded below on  $\mathscr N_{\la,\mu}$ and hence on $\mathscr N^+_{\la,\mu}$, there exists a minimizing sequence 
	$\{(u_m,v_m)\}\subset \mathscr N^+_{\la,\mu},$ such that
	{\small$$\DD\lim_{m\ra\infty}J_{\la,\mu}(u_m,v_m)=\DD\inf_{(u,v)\in \mathscr N^+_{\la,\mu} }J_{\la,\mu}(u,v).$$}
	As from Lemma \ref{lem6}, we have $J_{\la,\mu}$ is coercive on $\mathscr N^+_{\la,\mu}$, we get that $\{(u_m,v_m)\}$ is bounded on $E$.
	Hence there exists $(u_0,v_0)\in E,$ such that, passing to a sub-sequence 
	$$u_m\rightharpoonup u_0, ~~v_m\rightharpoonup v_0\text{~~~in~~} X_0\text{~~~~as~} m\ra\infty$$ 
	and hence using {Sobolev-type embedding result (Lemma \ref{prp 3.3})}, we have
	$$u_m\ra u_0~~\text{strongly in } L^{q(x)}(\Omega) \text{ and } L^{\alpha(x)+\beta(x)}(\Omega),$$
	$$v_m\ra v_0~~\text{strongly in } L^{q(x)}(\Omega)\text{ and } L^{\alpha(x)+\beta(x)}(\Omega),$$
	as $m\ra\infty.$ Therefore $u_m(x)\ra u_0(x)$ and  $v_m(x)\ra v_0(x)$ a.e. in $\Omega$ as $m\ra\infty.$ 
	Now by applying Lemma \ref{prp2} and Dominated convergence theorem, one can check that
	{\small\begin{align}\label{11}
		&\DD\lim_{m\rightarrow\infty}\int_{\Omega}a(x)|u_m|^{q(x)}dx=
		\int_{\Omega}a(x)|u_0|^{q(x)}dx,\;\;\DD\lim_{m\rightarrow\infty}\int_{\Omega}b(x)|v_m|^{q(x)}dx=
		\int_{\Omega}b(x)|v_0|^{q(x)}dx,
		\end{align}}
	{\small	\begin{align}\label{N1}
		&\DD\lim_{m\rightarrow\infty}\int_{\Omega}\frac{a(x)|u_m|^{q(x)}}{q(x)}dx=\int_{\Omega}\frac{a(x)|u_0|^{q(x)}}{q(x)}dx,
		\;\DD\lim_{m\rightarrow\infty}\int_{\Omega}\frac{b(x)|v_m|^{q(x)}}{q(x)}dx=\int_{\Omega}\frac{b(x)|v_0|^{q(x)}}{q(x)}dx.
		\end{align}}
	Also by Lemma \ref{convex} and  Lemma \ref{cor} (see Appendix), we have
	{\small\begin{align}\label{111}   
		\DD\lim_{m\rightarrow\infty}R(u_m,v_m)= R(u_0,v_0)\text{ and } 
		\DD\lim_{m\rightarrow\infty} \int_{\Omega}\frac{ c(x)|u_m|^{\alpha(x)}|v_m|^{\beta(x)}}{\alpha(x)+\beta(x)}dx=
		\int_{\Omega}\frac{c(x)|u_0|^{\alpha(x)}|v_0|^{\beta(x)}}{\alpha(x)+\beta(x)}dx,
		\end{align}} respectively.
	We claim that $(u_0,v_0)\not\equiv(0,0).$
	Note that $Q(u_0,v_0)>0.$ Indeed,
	if not then from \eqref{11},
	{\small\begin{align}\label{01}Q(u_m,v_m)\ra Q(u_0,v_0)=0~as~ m\ra \infty.\end{align}}
	Since $(u_m,v_m)\in \mathscr N^+_{\la,\mu}$, 
	using  \eqref{energy} and \eqref{1.0}, we get
	{\small\begin{align*}
	J_{\la,\mu}(u_m,v_m)
	&\geq \Big(\frac{1}{p^+}-\frac{1}{\alpha^-+\beta^-}\Big) P(u_m,v_m)-\Big(\frac{1}{q^-}-\frac{1}{\alpha^-+\beta^-}\Big)Q(u_m,v_m).
	\end{align*}}
	Now letting $m\ra\infty$ in the both side of the above expression and using \eqref{01}, we obtain 
	\begin{align}\label{88}\lim_{m\ra \infty} J_{\la,\mu}(u_m,v_m)\geq0.
	\end{align}
	But  Lemma \ref{lem7} gives that $\DD\lim_{m\ra \infty}	J_{\la,\mu}(u_m,v_m)=\DD\inf_{(u,v)\in\mathscr N^+_{\la,\mu}}J_{\la,\mu}(u,v)<0,$ 
	which contradicts \eqref{88}.
	Hence the claim is proved and we get that $(u_0,v_0)\in E\setminus\{(0,0)\}.$ 	 
	Next we claim that $$u_m\ra u_0  \text{ and }
	v_m\ra v_0 ~\text{strongly in ~} X_0\text{ as }m\ra\infty.$$ 
	Supposing the contrary, we have $u_m\nrightarrow u_0$ or $v_m\nrightarrow v_0$ strongly in $X_0$ as $m\ra\infty.$ 
	Therefore using Lemma \ref{lem 3.2} and Brezis-Lieb lemma (\cite{bl}), it follows that
{\small	\begin{align}\label{New1}
	& \text{either} \int_{\RR^N\times\RR^N}\frac{1}{p(x,y)}
	\frac{|
		u_0(x)-u_0(y)|^{p(x,y)}}
	{|x-y|^{N+sp(x,y)}}dxdy\nonumber\\&~~~~~~~~~~~~~~~~~~~~~~~~~~~~~~~~~
	<\DD\liminf_{m\ra\infty}\int_{\RR^N\times\RR^N}
	\frac{1}{p(x,y)}
	\frac{|u_m(x)-u_m(y)|^{p(x,y)}}{|x-y|^{N+sp(x,y)}}dxdy\nonumber\\
	&\text{~or~~}\int_{\RR^N\times\RR^N}\frac{1}{p(x,y)}
	\frac{|
		v_0(x)-v_0(y)|^{p(x,y)}}{|
		x-y|^{N+sp(x,y)}}dxdy\nonumber\\&~~~~~~~~~~~~~~~~~~~~~~~~~~~~~~~~~~
	<\DD\liminf_{m\ra\infty}\int_{\RR^N\times\RR^N}
	\frac{1}{p(x,y)}
	\frac{|v_m(x)-v_m(y)|^{p(x,y)}}{|x-y|^{N+sp(x,y)}}dxdy.
	\end{align}}
	Thus combining  \eqref{N1}, \eqref{111} and \eqref{New1}, from \eqref{energy}, we obtain {\begin{align}\label{123} 
		&\lim_{m\ra\infty} J_{\la,\mu}(u_m,v_m)\nonumber\\&=\DD\liminf_{m\ra\infty}\bigg[\int_{\RR^N\times\RR^N}\frac{1}{p(x,y)}
		\bigg\{\frac{|
			u_m(x)-u_m(y)|^{p(x,y)}}{|
			x-y|^{N+sp(x,y)}}+\frac{|
			v_m(x)-v_m(y)|^{p(x,y)}}{|
			x-y|^{N+sp(x,y)}}\bigg\}dxdy\nonumber\\ &~~~-\int_{\Omega}\frac{1}{q(x)}\Big(\la a(x)| u_m|^{q(x)}+
		\mu b(x)| v_m|^{q(x)}\Big) dx -\int_{\Omega}\frac{1}{\alpha(x)+\beta(x)} c(x)| u_m|^{\alpha(x)}| v_m| ^{\beta(x)} dx. \bigg]\nonumber\\
		&\geq\DD\liminf_{m\ra\infty}\int_{\RR^N\times\RR^N}\frac{1}{p(x,y)}\frac{|
			u_m(x)-u_m(y)|^{p(x,y)}}{|
			x-y|^{N+sp(x,y)}}dxdy\nonumber\\&~~~~~~+\DD\liminf_{m\ra\infty}\int_{\RR^N\times\RR^N}\frac{1}{p(x,y)}\frac{|
			v_m(x)-v_m(y)|^{p(x,y)}}{|
			x-y|^{N+sp(x,y)}}dxdy\nonumber\\ &~~~~~~-\DD\lim_{m\ra\infty}\int_{\Omega}\frac{1}{q(x)}\Big(\la a(x)| u_m|^{q(x)}+
		\mu b(x)| v_m|^{q(x)}\Big) dx\nonumber\\ &~~~~~~~-\DD\lim_{m\ra\infty}\int_{\Omega}\frac{1}{\alpha(x)+\beta(x)}c(x)| u_m|^{\alpha(x)}| v_m| ^{\beta(x)} dx\nonumber\\
		&>\int_{\RR^N\times\RR^N}\frac{1}{p(x,y)}\frac{|
			u_0(x)-u_0(y)|^{p(x,y)}}{|
			x-y|^{N+sp(x,y)}}dxdy\nonumber+\int_{\RR^N\times\RR^N}\frac{1}{p(x,y)}\frac{|
			v_0(x)-v_0(y)|^{p(x,y)}}{|
			x-y|^{N+sp(x,y)}}dxdy\nonumber\\ &~~~-\int_{\Omega}\frac{1}{q(x)}\Big(\la a(x)| u_0|^{q(x)}+
		\mu b(x)| v_0|^{q(x)}\Big) dx -\int_{\Omega}\frac{1}{\alpha(x)+\beta(x)}c(x)| u_0|^{\alpha(x)}| v_0| ^{\beta(x)} dx\nonumber\\
		&=  J_{\la,\mu}(u_0,v_0)
		\end{align}}
	\noindent Now using Lemma \ref{lem3} $(ii)$, for $(u_0,v_0)\in E\setminus\{(0,0)\},$ there exists a positive real number $t_0^+(u_0,v_0)$ 
	such that $(t_0^+u_0, t_0^+v_0)\in\mathscr N^+_{\la,\mu}.$  Again using the assumption $u_m\nrightarrow u_0$ or $v_m\nrightarrow v_0$ strongly in $X_0$, we have
	{\small\begin{align}\label{ineq1}
	\rho_{X_0}(t_0^+u_0)<\DD\liminf_{m\ra\infty}\rho_{X_0}(t_0^+u_m)~ \text{or}~ \rho_{X_0}(t_0^+v_0)<\DD\liminf_{m\ra\infty}\rho_{X_0}(t_0^+v_m).
	\end{align} }
	Also by Lemma \ref{prp2} and Dominated convergence theorem, we get 
	{\small\begin{align}\label{n01}
	Q(t_0^+u_0,t_0^+v_0)=\DD\lim_{m\ra\infty}Q(t_0^+u_m,t_0^+v_m)
	\end{align}}
	and by Lemma \ref{convex}(see Appendix),
	{\small\begin{align}\label{n02}
	R(t_0^+u_0,t_0^+v_0)=\DD\lim_{m\ra\infty}R(t_0^+u_m,t_0^+v_m).
	\end{align}}
	Taking into account \eqref{n01}, \eqref{n02} and \eqref{ineq1}, from \eqref{1.2}, we deduce 
	{\small	\begin{align}\label{1.25}
		&\DD\lim_{m\ra\infty}\varphi'_{u_m,v_m}(t_0^+)\nonumber\\&=\DD\liminf_{m\ra\infty}\bigg[\int_{\RR^N\times\RR^N}{(t_0^+)^{p(x,y)-1}}\bigg\{\frac{|
			u_m(x)-u_m(y)|^{p(x,y)}}{|
			x-y|^{N+sp(x,y)}}+\frac{|
			v_m(x)-v_m(y)|^{p(x,y)}}{|
			x-y|^{N+sp(x,y)}}\bigg\}dxdy\nonumber\\&~-\int_{\Omega}{(t_0^+)^{q(x)-1}}\Big(\la a(x)| u_m|^{q(x)}+
		\mu b(x)| v_m|^{q(x)}\Big) dx -\int_{\Omega}{(t_0^+)^{\alpha(x)+\beta(x)-1}}c(x)| u_m|^{\alpha(x)}| v_m| ^{\beta(x)} dx \bigg]\nonumber\\&\geq\frac {1}{t_0^+}\Bigg[\DD\liminf_{m\ra\infty}\rho_{X_0}(t_0^+u_m)+ \DD\liminf_{m\ra\infty}\rho_{X_0}(t_0^+v_m)-\lim_{m\ra \infty}Q(t_0^+u_m, t_0^+v_m)-\lim_{m\ra\infty}R(t_0^+u_m, t_0^+v_m)\Bigg]\nonumber\\
		&>\frac {1}{t_0^+}\Bigg[\rho_{X_0}(t_0^+u_0)+\rho_{X_0}(t_0^+v_m)-Q(t_0^+u_0, t_0^+v_0)-R(t_0^+u_0, t_0^+v_0)\Bigg]\nonumber\\
		&=\varphi'_{u_0,v_0}(t_0^+)=0.
		\end{align}}
	Thus for $m$ large enough $\varphi'_{u_m,v_m}(t_0^+)>0.$
	Since $(u_m,v_m)\in \mathscr N^+_{\la,\mu}$ for all $m\in \mathbb N,$ we have $\varphi'_{u_m,v_m}(1)=0$  and $\varphi''_{u_m,v_m}(1)>0.$ Then using  Lemma \ref{lem3} $(ii),$ we get $\varphi'_{u_m,v_m}(t)<0$ 
	for all $t\in(0,1).$ Therefore from \eqref{1.25}, we must have $t_0^+>1.$ Since $(t_0^+u_0, t_0^+v_0) \in \mathscr N^+_{\la,\mu},$ 
	again by Lemma \ref{lem3} $(ii),$ we obtain $\varphi_{u_0,v_0}(t)$ is monotone decreasing on $(0,t_0^+),$ therefore using \eqref{123},
	we conclude
	{\small$$J_{\la,\mu}(t_0^+u_0,t_0^+v_0)\leq J_{\la,\mu}(u_0,v_0)<\DD\lim_{m\ra\infty} J_{\la,\mu}(u_m,v_m)
	=\DD\inf_{(u,v)\in\mathscr N^+_{\la,\mu}} J_{\la,\mu}(u,v).$$} This is a contradiction as $(t_0^+u_0,t_0^+v_0)\in \mathscr N^+_{\la,\mu}. $ 
	Hence $(u_m,v_m)\ra (u_0,v_0) $ strongly in $E$ as $m\ra\infty$ and thus $(u_0,v_0)\in \mathscr N_{\la,\mu}$. Now
	as  Lemma \ref{lem5} gives that $\mathscr N^0_{\la,\mu}=\emptyset$ and by Lemma \ref{lem7}, 
	we have $J_{\la,\mu}(u_0,v_0)=\DD\lim_{m\ra\infty}J_{\la,\mu}(u_m,v_m)<0,$ we infer that 
	$(u_0,v_0)\in \mathscr N^+_{\la,\mu}$. 
	\\$(ii).$ Using Lemma \ref{lem1}, we can conclude $(u_0,v_0)$ is a solution of \eqref{mainprob}.
\end{proof}
\begin{proposition}\label{lem10}
If $\la+\mu<\delta_0,$  then $	J_{\la,\mu}$ has a  minimizer $(w_0,z_0)$ in $\mathscr N^-_{\la,\mu}$ such that the followings hold true.
\begin{itemize}
	\item [$(i).$ ] $J_{\la,\mu}(w_0,z_0)=\theta_{\la,\mu}^->0.$
	\item [$(ii).$]  $(w_0,z_0)$ is a non-semi trivial solution of the problem \eqref{mainprob}.	
\end{itemize}
\end{proposition}
\begin{proof}
Since by Lemma \ref{lem8}, $	J_{\la,\mu}$ is bounded below on  $\mathscr N^-_{\la,\mu}$, there exists a minimizing sequence $\{(w_m,z_m)\}\subset\mathscr N^-_{\la,\mu},$ 
such that {\small$$\DD\lim_{m\ra\infty}J_{\la,\mu}(w_m,z_m)=\DD\inf_{(u,v)\in \mathscr N^-_{\la,\mu} }J_{\la,\mu}(u,v).$$}
As from Lemma \ref{lem6}, we have  $J_{\la,\mu}$ is coercive,  we get that $\{(w_m,z_m)\}$ is bounded on $E$ and thus there exists $(w_0,z_0)\in E$
such that up to a sub-sequence  $(w_m,z_m)\rightharpoonup (w_0,z_0)$ weakly and by Sobolev-type embedding result 
(Theorem \ref{prp 3.3}), we have
$$w_m\ra w_0,~z_m\ra z_0~~\text{strongly in } L^{q(x)}(\Omega)\text{ and }L^{\alpha(x)+\beta(x)}(\Omega) \text{~as ~$m\ra\infty.$}$$
Therefore $w_m(x)\ra w_0(x)$ and  $z_m(x)\ra z_0(x)$ a.e. in $\Omega$ as $m\ra\infty.$ 
Now by Lemma \ref{prp2} and Dominated convergence theorem, we derive
{\small\begin{align}\label{21}
&\DD\lim_{m\rightarrow\infty}\int_{\Omega}a(x)|w_m|^{q(x)}dx=
\int_{\Omega}a(x)|w_0|^{q(x)}dx, 
\DD\lim_{m\rightarrow\infty}\int_{\Omega}b(x)|z_m|^{q(x)}dx=\int_{\Omega}b(x)|z_0|^{q(x)}dx.
\end{align}}
Also by Lemma \ref{convex} (see Appendix), we have
{\small\begin{align}\label{211}   
R(w_m,z_m)\ra R(w_0,z_0)~\text{as}~ m\ra\infty.
\end{align}}
{Next we have $(w_0,z_0)\not\equiv(0,0).$ Indeed, if $(w_0,z_0)=(0,0),$ from \eqref{211}, we obtain
{\small	\begin{align}\label{03}R(w_m,z_m)\ra R(w_0,z_0)=0 ~ \text{as~} m\ra\infty.\end{align}}
	Since $(w_m,z_m)\in \mathscr{N}_{\la,\mu}^-,$ 
	using \eqref{1.0} and Lemma \ref{lem7}, from \eqref{energy}, we deduce
	{\small\begin{align*}
	0<K&<J_{\la,\mu}(w_m,z_m)\leq \Big(\frac{1}{p^+}-\frac{1}{q^-}\Big)P(w_m,z_m)
	+\Big(\frac{1}{q^-}-\frac{1}{\alpha^-+\beta^-}\Big)R(w_m,z_m)+ o_m(1).
	\end{align*}}
	Now letting $m\ra\infty $ in both side of the above expression and using \eqref{03}, we have 
	{\small\begin{align*}0<K<\lim_{m\ra\infty}J_{\la,\mu}(w_m,z_m)\leq 0,
	\end{align*}} 
	which is a contradiction. Thus $(w_0,z_0)\in E\setminus\{(0,0)\}.$
	Now, if $Q(w_0,z_0)=0,$ we use Lemma \ref{lem3} $(i)$ and if $Q(w_0,z_0)>0,$ we use Lemma \ref{lem3} $(ii).$
	In both the cases,  there exists a positive real number $t_0^-=t_0^-(w_0,z_0)$ such that $(t_0^-w_0, t_0^-z_0)\in\mathscr N^-_{\la,\mu}.$}
Next we claim that $$w_m\ra w_0 ~\text{~strongly in } X_0\text{~ and~}
z_m\ra z_0 ~\text{~strongly in } X_0 \text{~as~} m\ra\infty.$$ Supposing the contrary, then     
$t_0^-w_m\nrightarrow t_0^- w_0$ or $t_0^-z_m\nrightarrow t_0^-z_0$ strongly in $X_0$ as $m\ra\infty$. This implies that 
{\small\begin{align}\label{001}
\text{ either}~\rho_{X_0}(t_0^-w_0)<\DD\liminf_{m\ra\infty}\rho_{X_0}(t_0^-w_m)~ \text{or}~ \rho_{X_0}(t_0^-z_0)<\DD\liminf_{m\ra\infty}\rho_{X_0}(t_0^-z_m).
\end{align}} 
Furthermore using the same assumption, we can have the following as in Proposition \ref{lem9}:
{\small\begin{align}\label{new1}
	&\text{either}  \int_{\RR^N\times\RR^N}\frac{1}{p(x,y)}
	\frac{|
		t_0^-w_0(x)-t_0^-w_0(y)|^{p(x,y)}}
	{|x-y|^{N+sp(x,y)}}dxdy\nonumber\\&~~~~~~~~~~~~~~~~~~~~~~~~~~~~~~~~~
	<\DD\liminf_{m\ra\infty}\int_{\RR^N\times\RR^N}
	\frac{1}{p(x,y)}
	\frac{|t_0^-w_m(x)-t_0^-w_m(y)|^{p(x,y)}}{|x-y|^{N+sp(x,y)}}dxdy\nonumber\\
	&\text{~or~~}\int_{\RR^N\times\RR^N}\frac{1}{p(x,y)}
	\frac{|
		t_0^-z_0(x)-t_0^-z_0(y)|^{p(x,y)}}{|
		x-y|^{N+sp(x,y)}}dxdy\nonumber\\&~~~~~~~~~~~~~~~~~~~~~~~~~~~~~~~~~~
	<\DD\liminf_{m\ra\infty}\int_{\RR^N\times\RR^N}
	\frac{1}{p(x,y)}
	\frac{|t_0^-z_m(x)-t_0^-z_m(y)|^{p(x,y)}}{|x-y|^{N+sp(x,y)}}dxdy.
	\end{align} }
Note that using Lemma \ref{prp2} and Dominated converges theorem, we can deduce
{\small\begin{align}\label{n4}
	\DD\lim_{m\rightarrow\infty}\int_{\Omega}\frac{1}{q(x)}\Big(\la a(x)|t_0^- w_m|^{q(x)}+
	\mu b(x)| t_0^-z_m|^{q(x)}\Big) dx\nonumber\\~~~~~~~~~~~~~~~~=\int_{\Omega}\frac{1}{q(x)}\Big(\la a(x)|t_0^- w_0|^{q(x)}+
	\mu b(x)| t_0^-z_0|^{q(x)}\Big) dx.
	\end{align}}
Also by Lemma \ref{cor}(see Appendix), we have
{\small\begin{align}\label{n5}
	\DD\lim_{m\ra\infty}\int_{\Omega}\frac{1}{\alpha(x)+\beta(x)}c(x)|t_0^- w_m|^{\alpha(x)}| t_0^-z_m| ^{\beta(x)} dx
	=\int_{\Omega}\frac{1}{\alpha(x)+\beta(x)}c(x)|t_0^- w_0|^{\alpha(x)}| t_0^-z_0| ^{\beta(x)} dx.
	\end{align}}
Thus combining  \eqref{new1}, \eqref{n4} and \eqref{n5}, from \eqref{energy}, we obtain 
{\small\begin{align}\label{3.25} 
	&\lim_{m\ra\infty} J_{\la,\mu}(t_0^-w_m,t_0^-z_m)\nonumber\\&=\DD\liminf_{m\ra\infty}\bigg[\int_{\RR^N\times\RR^N}\frac{1}{p(x,y)}
	\bigg\{\frac{|
		t_0^-w_m(x)-t_0^-w_m(y)|^{p(x,y)}}{|
		x-y|^{N+sp(x,y)}}+\frac{|
		t_0^-z_m(x)-t_0^-z_m(y)|^{p(x,y)}}{|
		x-y|^{N+sp(x,y)}}\bigg\}dxdy\nonumber\\ &~-\int_{\Omega}\frac{1}{q(x)}\Big(\la a(x)| t_0^-w_m|^{q(x)}+
	\mu b(x)| t_0^-z_m|^{q(x)}\Big) dx \nonumber\\&-\int_{\Omega}\frac{1}{\alpha(x)+\beta(x)} c(x)|t_0^- w_m|^{\alpha(x)}| t_0^-z_m| ^{\beta(x)} dx \bigg]\nonumber\\
	&\geq\DD\liminf_{m\ra\infty}\int_{\RR^N\times\RR^N}\frac{1}{p(x,y)}\frac{|
		t_0^-w_m(x)-t_0^-w_m(y)|^{p(x,y)}}{|
		x-y|^{N+sp(x,y)}}dxdy\nonumber\\&+\DD\liminf_{m\ra\infty}\int_{\RR^N\times\RR^N}\frac{1}{p(x,y)}\frac{|
		t_0^-z_m(x)-t_0^-z_m(y)|^{p(x,y)}}{|
		x-y|^{N+sp(x,y)}}dxdy\nonumber\\ &~-\DD\lim_{m\ra\infty}\int_{\Omega}\frac{1}{q(x)}\Big(\la a(x)|t_0^- w_m|^{q(x)}+
	\mu b(x)|t_0^- z_m|^{q(x)}\Big) dx \nonumber\\&-\DD\lim_{m\ra\infty}\int_{\Omega}\frac{1}{\alpha(x)+\beta(x)}c(x)|t_0^- w_m|^{\alpha(x)}| t_0^-z_m| ^{\beta(x)} dx\nonumber\\
	&>\int_{\RR^N\times\RR^N}\frac{1}{p(x,y)}\frac{|
		t_0^-w_0(x)-t_0^-w_0(y)|^{p(x,y)}}{|
		x-y|^{N+sp(x,y)}}dxdy\nonumber\nonumber\nonumber\\&+\int_{\RR^N\times\RR^N}\frac{1}{p(x,y)}\frac{|
		t_0^-z_0(x)-t_0^-z_0(y)|^{p(x,y)}}{|
		x-y|^{N+sp(x,y)}}dxdy\nonumber\\ &~-\int_{\Omega}\frac{1}{q(x)}\Big(\la a(x)|t_0^- w_0|^{q(x)}+
	\mu b(x)|t_0^- z_0|^{q(x)}\Big) dx -\DD\int_{\Omega}\frac{1}{\alpha(x)+\beta(x)}c(x)|t_0^- w_0|^{\alpha(x)}| t_0^-z_0| ^{\beta(x)} dx\nonumber\\
	&=  J_{\la,\mu}(t_0^-w_0,t_0^-z_0).
	\end{align}}
Again by using the strong convergence of $w_m\rightarrow w_0$ and $z_m\rightarrow z_0$ in $L^{q(x)}(\Omega)$, we deduce that
{\small\begin{align}\label{3.23}
	\lim_{m\ra\infty}
	Q (t_0^-w_m,t_0^-z_m)= Q (t_0^-w_0,t_0^-z_0)
	\end{align}}
and by Lemma \ref{convex}(see Appendix), we get
{\small\begin{align}\label{n3} 
	\lim_{m\ra\infty}R(t_0^-w_m,t_0^-z_m)= R(t_0^-w_0,t_0^-z_0).
	\end{align}}
Therefore using \eqref{001}, \eqref{3.23} and \eqref{n3}, from \eqref{1.2}, we deduce
{\small\begin{align}\label{2.25}
	&\DD\lim_{m\ra\infty}\varphi'_{w_m,z_m}(t_0^-)\nonumber\\
	&=\DD\liminf_{m\ra\infty}\bigg[\int_{\RR^N\times\RR^N}{(t_0^-)^{p(x,y)-1}}\bigg\{\frac{|
		w_m(x)-w_m(y)|^{p(x,y)}}{|
		x-y|^{N+sp(x,y)}}+\frac{|
		z_m(x)-z_m(y)|^{p(x,y)}}{|
		x-y|^{N+sp(x,y)}}\bigg\}dxdy\nonumber\\&~~~~~~~~~~~~~~~-\int_{\Omega}{(t_0^-)^{q(x)-1}}\Big(\la a(x)| w_m|^{q(x)}+
	\mu b(x)| z_m|^{q(x)}\Big) dx \nonumber\\&~~~~~~~~~~~~~~~~~~~-\int_{\Omega}{(t_0^-)^{\alpha(x)+\beta(x)-1}}c(x)| w_m|^{\alpha(x)}| z_m| ^{\beta(x)} dx \bigg]\nonumber\\&\geq\frac {1}{t_0^-}\Bigg[\DD\liminf_{m\ra\infty}\rho_{X_0}(t_0^-w_m)+ \DD\liminf_{m\ra\infty}\rho_{X_0}(t_0^-z_m)-\lim_{m\ra \infty}Q(t_0^-w_m, t_0^-z_m)\nonumber\\&~~~~~~~~~~~~~~~~~-\lim_{m\ra\infty}R(t_0^-w_m, t_0^-z_m)\Bigg]\nonumber\\
	&>\frac {1}{t_0^-}\Bigg[\rho_{X_0}(t_0^-w_0)+\rho_{X_0}(t_0^-z_m)-Q(t_0^-w_0, t_0^-z_0)-R(t_0^-w_0, t_0^-z_0)\Bigg]\nonumber\\
	&=\varphi'_{u_0,v_0}(t_0^-)=0.
	\end{align}}
For $m$ large enough $\varphi'_{w_m,z_m}(t_0^-)>0.$ Now since $(w_m,z_m)\in \mathscr N^-_{\la,\mu}$ for all $m\in \mathbb N,$ 
we have $\varphi'_{w_m,z_m}(1)=0$  and $\varphi''_{w_m,z_m}(1)<0$ for all $m\in \mathbb N.$
Now using the Lemma \ref{lem3}, we get $\varphi'_{w_m,z_m}(t)<0$ for all $t>1.$ Then from \eqref{2.25}, 
we must have $t_0^-<1.$ Since $(t_0^-w_0, t_0^-z_0) \in \mathscr N^-_{\la,\mu},$ again using Lemma \ref{lem3},
we obtain  1 is the global maximum point for $\varphi_{w_m,z_m}(t),$ therefore from \eqref{3.25}, we conclude
$$J_{\la,\mu}(t_0^-w_0,t_0^-z_0)
<\DD\lim_{m\ra\infty} J_{\la,\mu}(t_0^-w_m,t_0^-z_m)
\leq \lim_{m\ra\infty} J_{\la,\mu}(w_m,z_m)
=\DD\inf_{(u,v)\in\mathscr N^-_{\la,\mu}} J_{\la,\mu}(u,v).$$ This is a contradiction to the fact that $(t_0^-w_0,t_0^-z_0)\in \mathscr N^-_{\la,\mu}. $ 
Hence $(w_m,z_m)\ra (w_0,z_0) $ strongly in $E$ as $m\ra\infty$ and $(w_0,z_0)\in\mathscr{N}.$   
Also using the fact $\mathscr N^0_{\la,\mu}=\emptyset$ from Lemma \ref{lem5} and noticing that $J_{\la,\mu}(w_0,z_0)=\DD\inf_{(u,v)\in \mathscr N_{\la,\mu}^-}J_{\la,\mu}(u,v)>0,$  
we conclude that $(w_0,z_0)\in\mathscr N^-_{\la,\mu}.$
\\$(ii).$ Using Lemma \ref{lem1}, we can conclude $(w_0,z_0)$ is a solution of \eqref{mainprob}. 
Now we prove $(w_0,z_0)$ is not  semi-trivial, that is not of the form $(u,0)$ (or $(0,v)$). The proof follows as in \cite{chen-deng}.
If $(u,0)$ (or $(0,v)$) is a semi-trivial solution of problem \ref{mainprob}, then from \eqref{weakform}, we get 
\begin{align*}
\rho_{X_0}(u)=\int_{\RR^N\times\RR^N}\frac{|
	u(x)-u(y)|^{p(x,y)}}{|
	x-y|^{N+sp(x,y)}}dxdy =\la \int_{\Omega}  a(x)| u|^{q(x)} 
dx.
\end{align*} Therefore, 
\begin{align*}J_{\la,\mu}(u,0)&=\int_{\RR^N\times\RR^N}\frac{1}{p(x,y)}\frac{|
	u(x)-u(y)|^{p(x,y)}}{|
	x-y|^{N+sp(x,y)}}dxdy -\la \int_{\Omega} \frac{1}{q(x)}  a(x)| u|^{q(x)} 
dx\\
&\leq\frac{1}{p^-}\int_{\RR^N\times\RR^N}\frac{|
	u(x)-u(y)|^{p(x,y)}}{|
	x-y|^{N+sp(x,y)}}dxdy-\frac{\la}{q^+}\int_{\Omega}   
a(x)| u|^{q(x)} dx\\
&= \bigg(\frac{1}{p^-}-\frac{1}{q^+}\bigg)\rho_{X_0}(u)<0,
\end{align*}
since by Lemma \ref{lem8}, $J_{\la,\mu}(w_0,z_0)>0,$ we can conclude that $(w_0, z_0)$ is not semi-trivial.
\end{proof}
\noindent{\bf Proof of Theorem \ref{mainthm}.}  Define $\Lambda=\delta_0$ (as given in Section \ref{manifold}). {Let $(u_0,v_0)$ be as obtained in Proposition \ref{lem9}.
Now using Lemma \ref{lem*} and the fact that $(u_0,v_0)\in \mathscr N^+_{\la,\mu}$, for 
$(|u_0|,|v_0|)\in E\setminus \{(0,0)\}$, we have $Q(|u_0|, |v_0|)=Q(u_0, v_0)>0$, and
thus from Lemma \ref{lem3} $(ii)$, there exists $t_1>0$ such that $(t_1|u_0|,t_1|v_0|)\in\mathscr{N}_{\la,\mu}^+.$
This implies that 
\begin{equation}\label{M1}
0=\varphi'_{|u_0|,|v_0|}(t_1)\leq\varphi'_{u_0,v_0}(t_1).
\end{equation}
Now combining \eqref{M1} with the facts that
$(u_0,v_0)\in\mathscr{N}_{\la,\mu}^+$,
$\varphi'_{u_0,v_0}(1)=0$, 
and again using Lemma \ref{lem3} $(ii),$ we get $t_1\geq1.$ This implies that
$$J_{\la,\mu}(t_1|u_0|,t_1|v_0|)  \leq J_{\la,\mu}(|u_0|,|v_0|)\leq J_{\la,\mu}(u_0,v_0)=\DD\inf_{(u,v)\in\mathscr N^+_{\la,\mu}} J_{\la,\mu}(u,v).$$
Therefore we deduce that there exists a non-negative minimizer of $J_{\la,\mu}$ in $\mathscr{N}_{\la,\mu}^+$, which is a solution of problem \ref{mainprob} by Lemma \ref{lem1}.}
\par Next  we assert that there exists
a non-negative minimizer of $J_{\la,\mu}(w,z)$ in  $\mathscr N^-_{\la,\mu}$. Indeed   for $(|w_0|,|z_0|)\in E\setminus\{(0,0)\},$  by Lemma \ref{lem3}, there exists $t_2>0$ such that $(t_2|w_0|,t_2|z_0|)\in\mathscr{N}_{\la,\mu}^-,$ where $(w_0,z_0)$ is as given in Proposition \ref{lem10}. Since $(w_0,z_0)\in\mathscr{N}_{\la,\mu}^-,$ again by  Lemma \ref{lem3},
we get
$$J_{\la,\mu}(t_2|w_0|,t_2|z_0|)  \leq J_{\la,\mu}(t_2w_0,t_2z_0)\leq J_{\la,\mu}(w_0,z_0)=\DD\inf_{(u,v)\in\mathscr N^-_{\la,\mu}} J_{\la,\mu}(u,v).$$
Hence we get a non-negative minimizer of $J_{\la,\mu}$ in $\mathscr N^-_{\la,\mu},$ which is a solution of problem \ref{mainprob}, thanks to Lemma \ref{lem1}.\par From the above discussion,
we have that for all $0<\la+\mu<\Lambda,$ the problem \eqref{mainprob} admits two  non-trivial and non-negative solutions  
in $\mathscr N^+_{\la,\mu}$ and $\mathscr N^-_{\la,\mu},$ respectively.
Since $\mathscr N^+_{\la,\mu}\cap \mathscr N^-_{\la,\mu}=\emptyset,$ these solutions are distinct. Hence the proof is complete. $\hfill{\square}$
\section{ Appendix}
\begin{lemma}\label{convex}
Let $\{u_m\},\{v_m\}$ be any two bounded sequences in $X_0$ and
$c,\alpha, \beta$ be as in Theorem \ref{mainprob}. Then $$\lim_{m\ra \infty}\int_{\Omega}c(x)|u_m|^{\alpha(x)}|v_m|^{\beta(x)}dx= \int_{\Omega}c(x)|u|^{\alpha(x)}|v|^{\beta(x)}dx.$$ 
\end{lemma}
\begin{proof}
Since $\{u_m\},\{v_m\}$ are  bounded sequences in $X_0$ and $X_0$ is reflexive, up to sub-sequences $u_m\rightharpoonup u$ and $v_m\rightharpoonup v$ weakly in $X_0$ as $m\ra\infty.$
First we claim that 
\begin{align}\label{claim}
\lim_{m\ra\infty}\int_{\Omega} |u_m-u|^{\alpha(x)}|v_m-v|^{\beta(x)}dx=\lim_{m\ra \infty}\int_{\Omega}|u_m|^{\alpha(x)}|v_m|^{\beta(x)}dx-\int_{\Omega}|u|^{\alpha(x)}|v|^{\beta(x)}dx
\end{align}	 
For $t\in(0,1)$, we note that
\begin{align}\label{rel}
&\int_{\Omega}\int_{0}^{1}\alpha(x)|u_m-tu|^{\alpha(x)-2}(u_m-tu)u|v_m|^{\beta(x)}dxdt\nonumber\\&-\int_{\Omega}\int_{0}^{1}\beta(x)|u_m-u|^{\alpha(x)}|v_m-tv|^{\beta(x)-2}v(v_m-tv)dxdt\nonumber\\&
=\int_{\Omega}|u_m|^{\alpha(x)}|v_m|^{\beta(x)}dx-\int_{\Omega}|u_m-u|^{\alpha(x)}|v_m-v|^{\beta(x)}dx.
\end{align}
Denote{\small $$f_m(x,t):=|u_m-tu|^{\alpha(x)-2}(u_m-tu)|v_m|^{\beta(x)}\text{~and~} g_m(x,t):=|u_m-u|^{\alpha(x)}|v_m-tv|^{\beta(x)-2}(v_m-tv).$$}  Now from the given assumptions, we have
\begin{align}\label{a.e.}
\;\;\;	\left.\begin{array}{rl}
&f_m(x,t)\ra(1-t)^{\alpha(x)-1}|u|^{\alpha(x)-2}u|v|^{\beta(x)}\text{~a.e. ~in~ $\RR^N\times(0,1)$~as~} m\ra\infty, \\
&g_m(x,t)\ra 0 \text{~a.e. ~in~ $\RR^N\times(0,1)$~as~} m\ra\infty .
\end{array}
\right\}
\end{align}
Next, using H\"older's inequality and Sobolev-type embedding result (Theorem \ref{prp 3.3}), we obtain
\begin{align}\label{reln1}
&\int_{\Omega}\int_{0}^{1}|f_m|^{\frac{\alpha(x)+\beta(x)}{\alpha(x)+\beta(x)-1}}dxdt\nonumber\\
&\leq\||u_m-tu|^{\{(\alpha-1)(\frac{\alpha+\beta}{\alpha+\beta-1})\}(\cdot)}\|_{L^{\frac{\alpha(x)+\beta(x)-1}{\alpha(x)-1}}(\Omega\times(0,1))}
\||v_m|^{\beta(\cdot)}\|_{L^{\frac{\alpha(x)+\beta(x)-1}{\beta(x)}}(\Omega\times(0,1))}<M_1,\end{align} and
\begin{align}\label{reln2}
&\int_{\Omega}\int_{0}^{1}|g_m|^{\frac{\alpha(x)+\beta(x)}{\alpha(x)+\beta(x)-1}}dxdt\nonumber\\ &\leq\||u_m-u|^{\{(\alpha)(\frac{\alpha+\beta}{\alpha+\beta-1})\}(\cdot)}\|_{L^{\frac{\alpha(x)+\beta(x)-1}{\alpha(x)}}(\Omega\times(0,1))}\||v_m|^{(\beta-1)\frac{\alpha+\beta}{\alpha+\beta-1}(\cdot)}\|_{L^{\frac{\alpha(x)+\beta(x)-1}{\beta(x)-1}}(\Omega\times(0,1))}<M_2,
\end{align}
where $M_1, M_2$ are two positive constant independent of $m.$
Hence 
the sequences $\{f_m\}$, $\{g_m\}$ are uniformly bounded in  $L^{\frac{\alpha(x)+\beta(x)}{\alpha(x)+\beta(x)-1}}(\Omega\times(0,1))$  and thus
we have, up to sub-sequences
{\small\begin{align}\label{weakconv}
	\;\;\;	\left.\begin{array}{rl}
	&f_m\rightharpoonup(1-t)^{\alpha(x)-1}|u|^{\alpha(x)-2}u|v|^{\beta(x)}\text{~weakly~in~ $L^{\frac{{\alpha(x)+\beta(x)}}{\alpha(x)+\beta(x)-1}}(\Omega\times(0,1))$~as~} m\ra\infty, \\
	&g_m\rightharpoonup 0 \text{~weakly~in~ $L^{\frac{{\alpha(x)+\beta(x)}}{\alpha(x)+\beta(x)-1}}(\Omega\times(0,1))$~as~} m\ra\infty .
	\end{array}
	\right\}
	\end{align}}
Using \eqref{weakconv}, we deduce
\begin{align}\label{lim1}
\lim_{m\ra\infty}\int_{\Omega}\int_{0}^{1}\alpha(x)f_m u~ dx dt=\lim_{m\ra\infty}\int_{\Omega}\int_{0}^{1}\alpha(x) f u~ dxdt
=\lim_{m\ra\infty}\int_{\Omega}|u|^{\alpha(x)}|v|^{\beta(x)}dx,
\end{align}
and
\begin{align}\label{lim2}
\lim_{m\ra\infty}\int_{\Omega}\int_{0}^{1}\beta(x) g_m v~ dxdt=0.
\end{align}
Thus plugging \eqref{lim1} and \eqref{lim2} into \eqref{rel} we obtain \eqref{claim}. Note that from Theorem \ref{prp 3.3} and Lemma \ref{prp2},
we have $$\int_{\Omega} |u_m-u|^{\alpha(x)+\beta(x)}dx\ra 0 \text{~and~} \int_{\Omega}|v_m-v|^{\alpha(x)+\beta(x)}dx\ra 0 \text{~as~}m\ra\infty.$$
Now using the above and Young's inequality, we have
\begin{align}\label{lim}
&\int_{\Omega} |u_m-u|^{\alpha(x)}|v_m-v|^{\beta(x)}dx\nonumber\\
&\leq\int_{\Omega}\bigg\{\frac{\alpha(x)}{\alpha(x)+\beta(x)} |u_m-u|^{\alpha(x)+\beta(x)}+\frac{\alpha(x)}{\alpha(x)+\beta(x)} |v_m-v|^{\alpha(x)+\beta(x)}\bigg\}dx\nonumber\\
&\leq\frac{\alpha^+}{\alpha^-+\beta^-}\int_{\Omega} |u_m-u|^{\alpha(x)+\beta(x)}dx+\frac{\beta^+}{\alpha^-+\beta^-}\int_{\Omega} |v_m-v|^{\alpha(x)+\beta(x)}dx\nonumber\\
&\ra 0 \text{~ as~} m\ra\infty.
\end{align}
Thus inserting \eqref{lim} into \eqref{claim}, we obtain
\begin{align}\label{basic}
\lim_{m\ra\infty}\int_{\Omega}|u_m|^{\alpha(x)}|v_m|^{\beta(x)}dx= \int_{\Omega}|u|^{\alpha(x)}|v|^{\beta(x)}dx.
\end{align}
Now
\begin{align}\label{int}
&\left|\int_{\Omega}c(x)|u_m|^{\alpha(x)}|v_m|^{\beta(x)}dx- \int_{\Omega}c(x)|u|^{\alpha(x)}|v|^{\beta(x)}dx\right|
\nonumber\\&\leq\|c\|_{L^\infty(\Omega)}\int_{\Omega}\left||u_m|^{\alpha(x)}|v_m|^{\beta(x)}-|u|^{\alpha(x)}|v|^{\beta(x)}\right|dx.
\end{align}
Define $$w_m:=|u_m|^{\alpha(x)}|v_m|^{\beta(x)}+|u|^{\alpha(x)}|v|^{\beta(x)}-\left| |u_m|^{\alpha(x)}|v_m|^{\beta(x)}-|u|^{\alpha(x)}|v|^{\beta(x)}\right|\geq0.$$
Since $u_m(x)\ra u(x)$ and $v_m(x)\ra v(x)$ a.e. in $\RR^N$ as $m\ra\infty,$ we have $$w_m(x)\ra 2 |u(x)|^{\alpha(x)}|v(x)|^{\beta(x)} \text{~a.e.~in~$\RR^N$~as~$m\ra\infty$}.$$ Thus by Fatou's Lemma 
\begin{align}\label{fatou}
\DD\liminf_{m\ra\infty}\int_{\Omega}w_m(x)dx\geq 2 \int_{\Omega}|u|^{\alpha(x)}|v|^{\beta(x)}dx.
\end{align}
Again from \eqref{basic},
\begin{align}\label{final}
\DD\limsup_{m\ra\infty}\int_{\Omega}w_m(x)dx&\leq\lim_{m\ra \infty}\int_{\Omega}|u_m|^{\alpha(x)}|v_m|^{\beta(x)}dx
+\lim_{m\ra \infty}\int_{\Omega}|u|^{\alpha(x)}|v|^{\beta(x)}dx\nonumber\\&-\DD\limsup_{m\ra\infty}
\int_{\Omega}\left||u_m|^{\alpha(x)}|v_m|^{\beta(x)}dx-|u|^{\alpha(x)}|v|^{\beta(x)}\right|dx\nonumber\\
&=2\int_{\Omega}|u|^{\alpha(x)}|v|^{\beta(x)}dx-\DD\limsup_{m\ra\infty}
\int_{\Omega}\left||u_m|^{\alpha(x)}|v_m|^{\beta(x)}-|u|^{\alpha(x)}|v|^{\beta(x)}\right|dx
\end{align}
Combining \eqref{fatou} and \eqref{final}, we have $\DD\limsup_{m\ra\infty}\int_{\Omega}\left||u_m|^{\alpha(x)}|v_m|^{\beta(x)}-|u|^{\alpha(x)}|v|^{\beta(x)}\right|dx\leq0,$ that is $$ \lim_{m\ra \infty}\int_{\Omega}\left||u_m|^{\alpha(x)}|v_m|^{\beta(x)}-|u|^{\alpha(x)}|v|^{\beta(x)}\right|dx=0.$$ Thus combining the above together with \eqref{int}, we get our final result. 
\end{proof}
The next lemma follows similarly as Lemma \ref{convex} using the fact $\alpha,\beta\in C_+(\overline{\Omega}).$
\begin{lemma}\label{cor}
Let $\{u_m\},\{v_m\}$ be any two bounded sequences in $X_0$ and
$c,\alpha, \beta$ be as in Theorem \ref{mainprob}. Then $$\lim_{m\ra \infty}\int_{\Omega}\frac{1}{\alpha(x)+\beta(x)}c(x)|u_m|^{\alpha(x)}|v_m|^{\beta(x)}dx= \int_{\Omega}\frac{1}{\alpha(x)+\beta(x)}c(x)|u|^{\alpha(x)}|v|^{\beta(x)}dx.$$ 
\end{lemma} \hfill{$\square$}


\end{document}